\documentclass[12pt,a4paper
]{amsproc}  
\usepackage{graphicx,amssymb,stmaryrd,xspace}
\RequirePackage[raiselinks=false,colorlinks=true,citecolor=blue,urlcolor=black,linkcolor=blue,bookmarksopen=true]{hyperref}
\newcommand{\arxiv}[1]{\href{http://arxiv.org/pdf/#1}{arXiv:#1}}

\DeclareMathAlphabet{\mathbbe}{U}{bbold}{m}{n}

\newcommand{\simplexcategory}{\mathbbe{\Delta}}
\newcommand{\newnabla}{\rotatebox[origin=c]{180}{\(\simplexcategory\)}}

\DeclareMathAlphabet{\mathbfsf}{\encodingdefault}{\sfdefault}{bx}{n}

\newcommand{\ds}[1]{\mathbfsf{#1}}

\makeatletter

\renewenvironment{abstract}{%
  \ifx\maketitle\relax
    \ClassWarning{\@classname}{Abstract should precede
      \protect\maketitle\space in AMS document classes; reported}%
  \fi
  \global\setbox\abstractbox=\vtop \bgroup
    \normalfont\Small
    \list{}{\labelwidth\z@
      \leftmargin2pc \rightmargin\leftmargin
      \listparindent\normalparindent \itemindent\z@
      \parsep\z@ \@plus\p@
      
    }%
    \item[\hskip\labelsep\scshape\abstractname.]%
}{%
  \endlist\egroup
  \ifx\@setabstract\relax \@setabstracta \fi
}
\def\@setabstract{\@setabstracta \global\let\@setabstract\relax}
\def\@setabstracta{%
  \ifvoid\abstractbox
  \else
    \skip@20\p@ \advance\skip@-\lastskip
    \advance\skip@-\baselineskip \vskip\skip@
    \box\abstractbox
    \prevdepth\z@ % because \abstractbox is a vtop
  \fi
}

\makeatother

\usepackage{texdraw}
\input{txdtools}

\everytexdraw{%
	\drawdim pt \linewd 0.5 \textref h:C v:C
	\setunitscale 1
}
\newcommand{\onedot}{
  \bsegment
    \move (0 0) \fcir f:0 r:2
  \esegment
}

\usepackage[v2,all]{xy}

\newcommand{\dlpullback}[1][dl]{\save*!/#1-4ex/#1:(-1,1)@^{|-}\restore}
\newcommand{\urpullback}[1][ur]{\save*!/#1-4ex/#1:(-1,1)@^{|-}\restore}
\newcommand{\drpullback}[1][dr]{\save*!/#1-4ex/#1:(-1,1)@^{|-}\restore}

\usepackage{ stmaryrd }
\newcommand{\genmap}{\rightarrow\Mapsfromchar}

\newcommand{\Map}{\operatorname{Map}}
\newcommand{\iso}{\mathrm{iso}}

\newdir{ >}{{}*!/-7pt/@{>}}
\newcommand{\un}{\underline}
\newcommand{\unDelta}{\un\simplexcategory}
\newcommand{\Deltagen}{\simplexcategory_{\text{\rm gen}}}

\newcommand{\culf}{CULF\xspace}

\newcommand{\Dec}{\operatorname{Dec}}

\providecommand{\norm}[1]{\left| {#1}\right|}

\def\into{\hookrightarrow}

\providecommand{\kat}[1]{\text{\textbf{\textsl{#1}}}}

\newcommand{\rat}{\rightarrowtail}
\newcommand{\lat}{\leftarrowtail}

\newcommand{\upper}{\mathrm{upper}}
\newcommand{\low}{\mathrm{lower}}

\newcommand{\upperstar}{^{\raisebox{-0.25ex}[0ex][0ex]{\(\ast\)}}}

\newcommand{\lowershriek}{_!}

\newcommand{\isopil}{\stackrel{\raisebox{0.1ex}[0ex][0ex]{\(\sim\)}}%
			{\raisebox{-0.15ex}[0.28ex]{\(\rightarrow\)}}}

\newcommand{\tensor}{\otimes}
\newcommand{\op}{^{\text{{\rm{op}}}}}

\newcommand{\Set}{\kat{Set}}

\newcommand{\Grpd}{\kat{Grpd}}

\newcommand{\fatnerve}{\mathbf{N}}

\newcommand{\Q}{\mathbb{Q}}
\newcommand{\F}{\mathbb{F}}
\newcommand{\G}{\mathbb{G}}
\newcommand{\C}{\mathbb{C}}
\newcommand{\D}{\mathbb{D}}
\newcommand{\I}{\mathbb{I}}
\renewcommand{\P}{\mathbb{P}}
\newcommand{\R}{\mathbb{R}}

\usepackage{mathrsfs}

\newcommand{\CC}{\mathscr{C}}

\newcommand{\comma}{\raisebox{1pt}{$\downarrow$}}
\newcommand{\name}[1]{\ulcorner #1\urcorner}

\newcommand{\Hom}{\operatorname{Hom}}
\newcommand{\Fun}{\operatorname{Fun}}
\newcommand{\Aut}{\operatorname{Aut}}
\newcommand{\id}{\operatorname{id}}

\newcommand{\Ar}{\operatorname{Ar}}

\usepackage{amsthm}

\makeatletter
\let\sv@thm\@thm
\def\@thm{\let\indent\relax\sv@thm}
\renewenvironment{proof}[1][\proofname]{\par
  \pushQED{\qed}%
  \normalfont \topsep6\p@\@plus6\p@\relax
  \trivlist
  \itemindent\z@ % original has \normalparindent
  \item[\hskip\labelsep
        \scshape
    #1\@addpunct{.}]\ignorespaces
}{%
  \popQED\endtrivlist\@endpefalse
}
\makeatother

\newtheorem{lemma}{\bf Lemma}[section]
\newtheorem{prop}[lemma]{\bf Proposition}
\newtheorem{thm}[lemma]{\bf Theorem}
\newtheorem{theorem}[lemma]{\bf Theorem}
\newtheorem{cor}[lemma]{\bf Corollary}
\theoremstyle{definition}

\newtheorem{taller}[lemma]{$\!\!$}

\newenvironment{blanko}[1]%
{\begin{taller}{\normalfont\bfseries  #1}\normalfont}%
{\end{taller}}

\newenvironment{proof*}[1]{\begin{list}{\em #1 }%
{\setlength{\labelsep}{0mm}\setlength{\leftmargin}{0mm}%
\setlength{\labelwidth}{0mm}\setlength{\listparindent}{\parindent}%
\setlength{\parsep}{\parskip}\setlength{\partopsep}{0mm}}%
\item}{\qed\end{list}}

\thanks{%
  The first author %IMMA:
  was partially supported by grants 
  MTM2012-38122-C03-01,  % Casanelles
  2014-SGR-634,          % Alberich  
  MTM2013-42178-P,       % Casacuberta
  MTM2015-69135-P, and    % Casanelles
  MTM2016-76453-C2-2-P,  % Casacuberta
  the second author %JOACHIM:
  by 
  MTM2013-42293-P and % Castellana-Kock
  MTM2016-80439-P
  and the third author %ANDY:
  by
  MTM2013-42178-P and
  MTM2016-76453-C2-2-P}

\author{Imma G\'alvez-Carrillo}
\address{Departament de Matem\`atiques
      \\Universitat Polit\`ecnica de Catalunya
      \\Escola d'Enginyeria de Barcelona Est (EEBE) 
      \\Carrer Eduard Maristany 10-14\\08019 Barcelona\\Spain}
\email{m.immaculada.galvez@upc.edu}

\author{Joachim Kock}
\address{Departament de Matem\`atiques
       \\Universitat Aut\`onoma de Barcelona
       \\08193 Bellaterra (Barcelona), Spain}
\email{kock@mat.uab.cat}

\author{Andrew Tonks}
\address{Department of Mathematics\\ 
University of Leicester\\ 
University Road\\ 
Leicester LE1 7RH, UK}
\email{apt12@le.ac.uk}

\title{Decomposition spaces and restriction species}
\begin{document}

\begin{abstract}
  We show that Schmitt's restriction species (such as graphs, matroids,
  posets, etc.)\  naturally induce decomposition spaces (a.k.a.~unital
  $2$-Segal spaces), and that their associated coalgebras are an instance
  of the general construction of incidence coalgebras of decomposition
  spaces.  We introduce the notion of {\em directed restriction species}
  that subsume Schmitt's restriction species and also induce
  decomposition spaces. Whereas ordinary restriction species are presheaves
  on the category of finite sets and injections, directed restriction
  species are presheaves on the category of finite posets and convex maps.
  We also introduce the notion of {\em monoidal (directed)
  restriction species}, which induce monoidal decomposition spaces and
  hence bialgebras, most often Hopf algebras. Examples of this notion
  include rooted forests, directed graphs, posets, double posets, and many
  related structures.  A prominent instance of a resulting incidence
  bialgebra is the Butcher--Connes--Kreimer Hopf algebra of rooted trees.
  Both ordinary and directed restriction species are shown to be examples
  of a construction of decomposition spaces from certain cocartesian
  fibrations over the category of finite ordinals that are also cartesian
  over convex maps.  The proofs rely on some beautiful simplicial
  combinatorics, where the notion of convexity plays a key role.  The
  methods developed are of independent interest as techniques for
  constructing decomposition spaces.
\end{abstract}


% 18G30 Simplicial sets
% 16T10 Bialgebras
% 06A07 combinatorics of posets
% 06A11 algebraic aspects of posets
% 18-XX Category theory
% 55Pxx Homotopy Theory

\maketitle

\begin{center}
\begin{minipage}{120mm}	

\small

\tableofcontents

\normalsize

\end{minipage}
\end{center}

\setcounter{section}{-1}

\pagebreak

\addtocontents{toc}{\protect\setcounter{tocdepth}{1}}

%%%%%%%%%%%%%%%%%%%%%%%%%%%%%%%%%%%%%%%%%%%%%%%%%%
\section{Introduction}
%%%%%%%%%%%%%%%%%%%%%%%%%%%%%%%%%%%%%%%%%%%%%%%%%%

The notion of decomposition space was introduced in \cite{GKT:DSIAMI-1} as
a very general framework for incidence (co)algebras and M\"obius inversion.
Let us briefly recount the abstraction steps that led to this notion,
taking as starting point the classical theory of incidence algebras of
locally finite posets.  More extensive introductions can be found in
\cite{GKT:DSIAMI-1} and in \cite{GKT:ex}.  A very different motivation and
formulation of the notion is due to Dyckerhoff and
Kapranov~\cite{Dyckerhoff-Kapranov:1212.3563}.

The first step is the observation due to Leroux~\cite{Leroux:1975}, that
both the notions of locally finite poset (Rota
et.~al~\cite{JoniRotaMR544721,Rota:Moebius}) and monoid with the
finite decomposition property (Cartier--Foata~\cite{Cartier-Foata}) admit a
natural common generalisation in the notion of M\"obius category, and that
this setting allows for good functorial properties.

The next step is to observe that in many examples where
symmetries play a role, a more elegant treatment can be achieved by considering
groupoid-enriched categories instead of plain (set-enriched) 
categories, as illustrated in \cite{GalvezCarrillo-Kock-Tonks:1207.6404}.  This
involves a homotopical viewpoint, in which the algebraic identities arise as
homotopy cardinality of equivalences of groupoids, rather than just ordinary 
cardinality of bijections of sets.
At the same time it becomes clear that the algebraic structures can actually
be defined and manipulated at the objective level, postponing the act of taking
cardinality, and that structural phenomena can be seen at this level which are
not visible at the usual `numerical' level.  For example, at this level of
abstraction one can view the algebra of species under the Cauchy tensor product
as the incidence algebra of the symmetric monoidal category of finite sets and
bijections~\cite{GKT:ex}.  (The homotopy viewpoint induces one to consider even
$\infty$-groupoids~\cite{GKT:HLA,GKT:DSIAMI-1}, but this is not important 
in the present contribution.)

Finally, considering groupoid-enriched categories as simplicial groupoids via
the nerve construction led to the discovery~\cite{GKT:DSIAMI-1}
that the Segal condition, which essentially
characterises category objects among simplicial groupoids, is not actually
needed, and that a weaker notion suffices for the theory of incidence 
(co)algebras
and M\"obius inversion: this is the notion of decomposition space, which can be
seen as the systematic theory of decompositions, where categories are the
systematic theory of compositions.

While many coalgebras and bialgebras in combinatorics do arise
from (groupoid-enriched) categories, there are also many examples that can
easily be seen {\em not} to arise from such categories. Two 
prominent examples are the Schmitt Hopf algebra of graphs~\cite{Schmitt:1994} 
(also called the chromatic Hopf algebra~\cite{Aguiar-Bergeron-Sottile}), and
the Butcher--Connes--Kreimer Hopf algebra of rooted trees 
(see~\cite{Dur:1986} and~\cite{Connes-Kreimer:9808042}).
These two examples are reviewed below, where we shall see
that they cannot possibly arise directly from categories,
but that they do naturally come from decomposition spaces,
cf.~\cite{GKT:DSIAMI-1,GKT:ex}.
(They can be obtained indirectly from certain auxiliary categories, 
by means of a reduction step, cf.~D\"ur~\cite{Dur:1986}.)

The aim of the present paper is to fit these two examples into a large
class of decomposition spaces.  One may say there are two large classes of
decomposition spaces, but the first can be regarded as a special case of
the second.  The first is the class of decomposition spaces coming from
Schmitt's restriction species~\cite{Schmitt:hacs}---Schmitt already showed
that the Hopf algebra of graphs comes from a restriction species. 
While restriction species are presheaves on the category of finite sets
and injections, expressing the ability to decompose combinatorial 
structures, the new notion of directed restriction species expresses
decompositions compatible with an underlying partial order:

\medskip

\noindent{\em Definition.} A {\em directed restriction species} is a
presheaf on
the category of finite posets and convex maps.

\medskip

\noindent Ordinary restriction species can be regarded
as directed restriction species supported on
discrete posets.

We show that every directed restriction species defines a decomposition
space, and hence a coalgebra.  Instead of constructing these simplicial
objects by hand, we found it worth taking a slight detour through some more
abstract constructions.  On one hand, this serves to exhibit the general
principles behind the results, and on the other to develop machinery of
independent interest for the sake of constructing decomposition spaces.  We
route the construction through certain sesquicartesian fibrations over
$\un\simplexcategory$ (the category of finite ordinals, including the empty
ordinal): they are cocartesian fibrations which are furthermore cartesian
over convex maps, satisfying Beck--Chevalley, and subject to one further
condition which we refer to as the {\em iesq} (for `{\em
identity-extension-square}') condition.

The main results can now be organised as follows:

\medskip % similarly below, so that section 1 starts new page

\noindent {\bf Theorem.}
(Proposition~\ref{prop:DRSp->iesq} and Corollary~\ref{cor:RSp->iesq}.)
{\em  Restriction species and directed restriction species
naturally induce iesq sesquicartesian fibrations.}

\medskip

\noindent {\bf Theorem~\ref{thm:iesqsesqui->decomp}.} 
{\em Iesq sesquicartesian fibrations naturally induce decomposition spaces.}

\medskip

Together, and more precisely:

\medskip

\noindent {\bf Theorem.} 
(Theorems~\ref{thm:DRSp->decomp/C=ff} and \ref{thm:RSp->decomp/I=ff}.)
{\em There is a functor from restriction species to decomposition spaces
\culf over $\ds I$, and this functor is fully faithful.  Similarly there is
a functor from directed restriction species to decomposition spaces \culf
over $\ds C$, also fully faithful.}

\medskip

\noindent Here $\ds I$ is a certain decomposition space of layered finite
sets (\S\ref{sec:B}), and $\ds C$ is a certain decomposition space of
layered finite posets (\S\ref{sec:C}). For \culf functors, see \ref{culf} 
below.

\medskip

Many combinatorial structures which form (directed) restriction species
are closed under taking disjoint union in a way compatible with 
restrictions.  We capture this through the notion of monoidal 
directed restriction species (\ref{monoidaldirected}), and show:

\medskip

\noindent
{\bf Proposition~\ref{prop:mdrsp-mds}.} {\em Monoidal directed restriction
species naturally induce monoidal decomposition spaces and hence
bialgebras.}

\medskip

\noindent
Examples of this notion
  include rooted forests, directed graphs, posets, double posets, and many
  related structures.  A prominent instance of a resulting incidence
  bialgebra is the Butcher--Connes--Kreimer Hopf algebra of rooted trees.

\medskip

\noindent {\bf Note.}
  This paper was originally posted as Section~6 of the long
  manuscript {\em Decomposition spaces, incidence algebras 
  and M\"obius inversion}~\cite{GKT:1404.3202}, which has now been split into
  six papers, the first five being \cite{GKT:HLA,GKT:DSIAMI-1,GKT:DSIAMI-2,GKT:MI,GKT:ex}.
  The relevant definitions and results from these papers 
  (mostly~\cite{GKT:DSIAMI-1}) are reviewed below as needed, to render the 
  paper reasonably self-contained.

\bigskip

\noindent {\bf Acknowledgments.} We wish to thank Andr\'e Joyal and Mark Weber
for some very pertinent remarks, and apologise for not being able to follow them
through to their full depth in the present contribution. 

%%%%%%%%%%%%%%%%%%%%%%%%%%%%%%%%%%%%%%%%%%%%%%%%%%
\section{Decomposition spaces}
%%%%%%%%%%%%%%%%%%%%%%%%%%%%%%%%%%%%%%%%%%%%%%%%%%

In this section we briefly recall and motivate the notion of decomposition space.

\begin{blanko}{Incidence coalgebras of locally finite posets and categories.}
  Recall from Rota et al.~\cite{JoniRotaMR544721,Rota:Moebius}
  that for a locally finite poset, a coalgebra
  structure is induced on the vector space spanned by its intervals,
  with comultiplication given by
  $$
  \Delta( [x,y] ) = \sum_{m\in [x,y]} [x,m] \tensor [m,y].
  $$
  The local finiteness condition is precisely what ensures that the sum is 
  finite.  Coassociativity is a consequence of transitivity of the poset 
  relation.
  
  A poset can be regarded as a category in which there is one arrow
  from $x$ to $y$ if and only if $x\leq y$.  Thus intervals in a poset
  correspond to arrows in the category, and the incidence coalgebra construction
  generalises immediately to locally finite categories, as first observed by 
  Leroux~\cite{Leroux:1975}: the coalgebra has as
  underlying vector space the one spanned by the arrows, and the comultiplication
  is given by
  $$
  \Delta(f) = \sum_{b\circ a = f} a \tensor b  .
  $$
  Coassociativity follows from associativity of composition of arrows.
\end{blanko}

\begin{blanko}{Nerves, and an objective comultiplication.}
  The nerve of a category $\CC$ (e.g.~a poset) is the simplicial set $X: \simplexcategory\op 
  \to \Set$ whose $n$-simplices are sequences of $n$ composable arrows.
  This can be written formally as 
  $$
  X_n = \Fun([n],\CC)  ,
  $$
  where $\Fun([n],\CC)$ denotes just the {\em set} of functors $[n]\to\CC$.
  The face maps $d_i: X_{n+1} \to X_n$ compose the two consecutive arrows at 
  the $i$th object (for the inner face maps, $0 < i < n$) or project away the first or last arrow
  in the sequence (for the outer face maps, $i=0$ or $i=n$).  The comultiplication formula can now be
  seen at the objective level of the arrows themselves (not the vector space
  spanned by them) as given by the canonical span
  $$
  X_1 \stackrel{d_1}{\longleftarrow} X_2 \stackrel{(d_2,d_0)}{\longrightarrow} X_1 
  \times X_1
  $$
  by pullback along $d_1$ and composing along $(d_2,d_0)$.  Indeed the fibre
  of $d_1$ over an arrow $f\in X_1$ is the set of composable 
  pairs with composite $f$,
  and $(d_2,d_0)$ then returns the two constituents.
  Properly formalising this construction involves working with the slice
  category $\Set_{/X_1}$ instead of the vector space spanned by $X_1$, and the
  comultiplication is then a functor rather than just a function.  The classical
  viewpoint can be recovered by taking cardinality of the sets involved.
\end{blanko}

\begin{blanko}{Groupoids and homotopy viewpoints.}
  In practice one is often interested in combinatorial objects up to
  isomorphism, but at the same time wants to keep track of automorphisms.
  This can be accomplished elegantly by working with groupoids
  instead of sets, provided the homotopy viewpoint is taken consistently.
  The classical viewpoint is recovered by taking 
  homotopy cardinality, and all constructions should be performed in a
  homotopy invariant way.  In particular, all pullbacks must be homotopy 
  pullbacks, since this is the homotopy invariant notion.  
  \begin{quote}
	  {\em Throughout,
  when we say pullback, we refer to the homotopy pullback.}
  \end{quote}
  Strict pullbacks are {\em not} in general homotopy invariant, except
  if one of the maps pulled back along is an
  iso-fibration; this will be exploited occasionally.  Similarly, when we
  talk about simplicial groupoids, we must allow pseudo-functors
  $\simplexcategory\op\to\Grpd$ instead of just strict functors, since this
  is the homotopy invariant notion.  Most of our simplicial groupoids will
  actually happen to be strict, though, as is the case with fat nerves:
\end{blanko}

\begin{blanko}{Fat nerve.}  
  Starting with a small category $\CC$, instead of working with its
  ordinary nerve as above, one considers instead its
  {\em fat nerve}. This is a simplicial groupoid
  $X = \fatnerve \CC
  : \simplexcategory\op \to \Grpd$ rather
  than a simplicial set, and is
  defined formally by
  $$
  X_n = \Map([n],\CC) ,
  $$
  the groupoid whose objects are functors $[n]\to \CC$ (i.e.~$n$-sequences of 
  arrows), and whose morphisms are invertible natural transformations between 
  them.  This means that we keep track of the fact that two arrows $f$ and $g$
  in $\CC$ may be isomorphic by way of a commutative square
  $$\xymatrix{
     \cdot  \ar[r]^f\ar[d]_\simeq & \cdot \ar[d]^\simeq \\
     \cdot \ar[r]_g & \cdot
  }$$
  and similarly for $n$-sequences.
  The fat nerve constitutes a functor from categories to simplicial 
  groupoids, and this functor is fully faithful.
  
  The comultiplication formula resulting from the span construction now
  concerns isoclasses of arrows, and the sum is over isoclasses of
  factorisations.  In practice this is precisely what one wants.  For
  example, if $\CC$ is the category of finite sets and surjections, the
  incidence coalgebra resulting from the fat nerve is the Fa\`a di Bruno
  coalgebra~\cite{JoyalMR633783}.  Recovering the classical setting now
  involves homotopy cardinality of groupoids rather than cardinality of
  sets---this is just a question of taking the isomorphisms into account
  properly.  This will be recalled below in \ref{card}.
\end{blanko}

\begin{blanko}{Decomposition spaces.}
  It turns out that simplicial groupoids other than fat nerves of categories
  induce coalgebras.  Fat nerves of categories can be characterised (in part)
  by the Segal condition, which can be stated as requiring all squares of the form
  $$\xymatrix{
     X_{n+1} \drpullback \ar[r]^{d_0}\ar[d]_{d_{n+1}} & X_n \ar[d]^{d_n} \\
     X_n \ar[r]_{d_0} & X_{n-1}
  }$$
  to be pullbacks.  The most important one is
  $$\xymatrix{
     X_2 \drpullback \ar[r]^{d_0}\ar[d]_{d_2} & X_1 \ar[d]^{d_1} \\
     X_1 \ar[r]_{d_0} & X_0
  }$$
  which says that $X_2$ can be identified with the groupoid $X_1 \times_{X_0} 
  X_1$ of composable pairs of arrows.  The Segal condition thus
  expresses the ability to compose.
  
  The decomposition-space axiom, which is weaker, stipulates that certain
  other squares are pullbacks, the most important cases being
  $$\centerline{
  \xymatrix{
     X_3 \drpullback \ar[r]^{d_2}\ar[d]_{d_0} & X_2 \ar[d]^{d_0} \\
     X_2 \ar[r]_{d_1} & X_1
  }
  \qquad 
  \qquad
  \xymatrix{
     X_3 \drpullback \ar[r]^{d_1}\ar[d]_{d_3} & X_2 \ar[d]^{d_2} \\
     X_2 \ar[r]_{d_1} & X_1   .
  }
  }$$
  We refer to \cite{GKT:ex} for an explanation of the combinatorial meaning
  of this condition and a picture. It can be interpreted as the 
  expression of the ability to {\em decompose}.
\end{blanko}

To define more formally what a decomposition space is---and to
construct them---we need some simplicial technicalities.

\begin{blanko}{Generic and free maps (active and inert 
	maps~\cite{Lurie:HA}).}\label{generic-and-free}
  The category $\simplexcategory$ of nonempty finite ordinals $[n] = 
  \{0,1,\ldots,n\}$ and monotone maps
  has a so-called generic-free factorisation system (a general categorical
  notion, important in monad theory \cite{Weber:TAC13,Weber:TAC18}).  An arrow
  $a: [m]\to [n]$ in $\simplexcategory$ is \emph{generic} (also 
  called {\em active}) when it preserves
  end-points, $a(0)=0$ and $a(m)=n$; we use the special arrow symbol 
  $\genmap$ to denote generic maps. An arrow
  $a: [m]\to [n]$ in $\simplexcategory$ is \emph{free} (also called 
  {\em inert}) if it is distance
  preserving, $a(i+1)=a(i)+1$ for $0\leq i\leq m-1$; we use the 
  special arrow symbol $\rat$.  The generic maps are
  generated by the codegeneracy maps $s^i : [n{+}1] \to [n]$ and by the {\em
  inner} coface maps $d^i : [n{-}1]\to [n]$, $0 < i < n$, while the free maps are
  generated by the {\em outer} coface maps $d^\bot := d^0$ and $d^\top:= d^n$.
  Every morphism in $\simplexcategory$ factors uniquely as a generic map
  followed by a free map.  Furthermore, it is a basic fact \cite{GKT:DSIAMI-1}
  that generic and free maps in $\simplexcategory$ admit pushouts along each
  other, and the resulting maps are again generic and free.
\end{blanko}

\begin{blanko}{Decomposition spaces \cite{GKT:DSIAMI-1}.}
  A simplicial groupoid $\ds X:\simplexcategory\op\to\Grpd$ is called a {\em decomposition space}
  when it takes generic-free pushouts in $\simplexcategory$ to pullbacks.

  The notion is equivalent to the unital $2$-Segal spaces of Dyckerhoff and
  Kapranov~\cite{Dyckerhoff-Kapranov:1212.3563}, formulated in terms of
  triangulations of polygons.  Their work shows that the notion is of
  interest well beyond combinatorics.
\end{blanko}

\begin{theorem}
  \cite{GKT:DSIAMI-1}
  If $\ds X:\simplexcategory\op\to\Grpd$ is a decomposition space, the span construction
  above induces on $\Grpd_{/X_1}$ the structure of a coassociative and counital
  coalgebra (up to coherent equivalence).  Upon taking homotopy cardinality
  (in suitably finite situations, cf.~\ref{finite} below),
  this yields a coalgebra in the classical sense.
\end{theorem}

The fat nerve of a category is always a
decomposition space.  Since the Segal-axiom squares are not special cases of the
decomposition-space axioms, this requires proof, but it is not a deep result
\cite{GKT:DSIAMI-1}.
Intuitively, the reason is that in situations where one
can compose (that is, in a category), one can always decompose, by summing 
over all
possible ways an object could have arisen by composition.

\begin{blanko}{Finiteness conditions (cf.~\cite{GKT:DSIAMI-2}).}\label{finite}
  Various finiteness conditions are important for various reasons.
  They tend to be satisfied in examples coming from combinatorics, and
  we shall establish them for all restriction species and directed
  restriction species.  Let us briefly comment on these conditions.
    
  In order to be able to take homotopy cardinality to get a coalgebra
  in vector spaces, it is necessary to assume that $X$ is {\em locally
  finite} (cf.~\cite[\S7]{GKT:DSIAMI-2}).  This means first of all
  that $X_1$ is a locally finite groupoid (i.e.~has finite
  automorphism groups), and second that each generic map is a finite
  map (i.e.~has finite fibres).  For a decomposition space $X$, this
  can be measured on the two maps
  $$
  X_0 \stackrel{s_0}\to X_1 \stackrel{d_1}\leftarrow X_2.
  $$

  For the comultiplication formula to be free of denominators, another
  condition is required, namely that $X$ must be {\em locally
  discrete} (cf.~\cite[\S 1.4]{GKT:ex}), which for a decomposition
  space amounts to the two displayed maps having discrete fibres.

  In order to have a M\"obius inversion formula, yet another finiteness
  condition is needed, which refers to a notion of non-degeneracy
  which is meaningful for {\em complete} decomposition spaces
  (cf.~\cite[\S2]{GKT:DSIAMI-2}), i.e.~those for which $s_0$ is mono.
  The condition is to have {\em locally finite length}, and it means
  (cf.~\cite[\S6]{GKT:DSIAMI-2}) that for each $a\in X_1$ there is an
  upper bound on the $n$ for which the map $X_n\to X_1$ has
  non-degenerate elements in the fibre.  See op cit.~for
  precision---the upshot is that there are only finitely many ways of
  splitting an object into non-degenerate pieces.
\end{blanko}

\begin{blanko}{Homotopy cardinality.}\label{card}
  Assuming local finiteness, the groupoid-level incidence coalgebra
  yields a vector-space level coalgebra by taking homotopy
  cardinality.  We refer to \cite{GKT:HLA} for the full story (in the
  setting of $\infty$-groupoids) and to \cite{GKT:ex} for some
  introduction geared towards combinatorics.  Very briefly, the
  homotopy cardinality of a groupoid $X$ is defined to be
  $\sum_{x\in\pi_0 X} \frac{1}{\norm{\Aut(x)}}$.  The groupoid slice
  $\Grpd_{/S}$ is the objective counterpart of the vector space
  $\Q_{\pi_0 S}$ spanned by the symbols $\delta_s$ denoting isoclasses
  of objects in $S$.  The cardinality of an object $X \to S$ is then
  the formal linear combination $\sum_{s\in \pi_0 S}
  \frac{\norm{X_s}}{\norm{\Aut(s)}}\, \delta_s$, where $\norm{X_s}$ is the
  homotopy cardinality of the homotopy fibre $X_s$.

  If the groupoids
  involved are just sets, the automorphism groups are trivial,
  and the notion reduces to ordinary cardinality.  Building the
  automorphism groups into the definition ensures it
  behaves well with respect to all the important operations, such as
  products and sums, (homotopy) pullbacks and (homotopy) fibres, etc.
\end{blanko}

\begin{blanko}{CULF functors.}\label{culf}
  The relevant notion of morphism between decomposition spaces is that
  of \culf functor \cite{GKT:DSIAMI-1}: \culf functors between
  decomposition spaces induce coalgebra homomorphisms.  A simplicial
  map is called {\em ULF} (unique lifting of factorisations) if it is
  cartesian on generic face maps, and it is called {\em conservative}
  if cartesian on degeneracy maps.  We say {\em \culf} for
  conservative and ULF, that is, cartesian on all generic maps.
  
  Since {\culf}ness refers to generic maps, just as the finiteness 
  conditions just stated, we have the following useful result.
\end{blanko}

\begin{lemma}\label{lem:culf-local}
  Let $P$ denote a property of decomposition spaces which is measured on
  generic maps (such as being locally discrete or of locally finite
  length).  Then if $F: \ds Y \to \ds X$ is \culf and $\ds X$ has property
  $P$, then also $\ds Y$ has property $P$.  This is also the case for the
  property of being locally finite, except we must check additionally that
  $Y_1$ is locally finite.
\end{lemma}

In fact, also:
\begin{lemma}\label{lem:CULF}
  A simplicial groupoid \culf over a decomposition space is itself a 
  decomposition space.
\end{lemma}

\begin{blanko}{Monoidal decomposition spaces and bialgebras.}
  There is a natural notion of {\em monoidal decomposition space}
  \cite{GKT:DSIAMI-1}, leading to bialgebras.  Briefly, it is a
  decomposition space $\ds X$ equipped with a functor $\tensor : \ds X
  \times \ds X \to \ds X$ required to be a monoidal structure, and required
  to be \culf.  The homotopy cardinality of this monoidal structure is an
  algebra structure, and the \culf condition ensures the compatibility with
  the coalgebra structure to result altogether in a bialgebra.  This is
  important in most applications to combinatorics, where almost always this
  monoidal structure, and hence the algebra structure, is given by disjoint
  union.  In the present contribution we focus mostly on the
  comultiplication, but comment on monoidal structure in
  \ref{monoidalrestr}--\ref{prop:monoidalrestr} and
  \ref{monoidaldirected}--\ref{prop:mdrsp-mds}.
\end{blanko}

\begin{blanko}{Decalage.}
  (See \cite{Illusie2}.)  Given a simplicial groupoid
  $\ds X$ as the top row in the following diagram, the {\em lower 
  Dec},
  $\Dec_\bot(\ds X)$, is a new simplicial groupoid (the bottom row in the diagram)
  obtained by deleting $X_0$ and shifting everything one place down, deleting
  also all $d_0$ face maps and all $s_0$ degeneracy maps.  It comes equipped
  with a simplicial map, called the {\em dec map},
  $d_\bot:\Dec_\bot(\ds X)\to \ds X$ given by the original $d_0$:

$$
\xymatrix@C+1em{
X_0  
\ar[r]|(0.55){s_0} 
&
\ar[l]<+2mm>^{d_0}\ar[l]<-2mm>_{d_1} 
X_1  
\ar[r]<-2mm>|(0.6){s_0}\ar[r]<+2mm>|(0.6){s_1}  
&
\ar[l]<+4mm>^(0.6){d_0}\ar[l]|(0.6){d_1}\ar[l]<-4mm>_(0.6){d_2}
X_2 
\ar[r]<-4mm>|(0.6){s_0}\ar[r]|(0.6){s_1}\ar[r]<+4mm>|(0.6){s_2}  
&
\ar[l]<+6mm>^(0.6){d_0}\ar[l]<+2mm>|(0.6){d_1}\ar[l]<-2mm>|(0.6){d_2}\ar[l]<-6mm>_(0.6){d_3}
X_3 
\ar@{}|\cdots[r]
&
\\
\\
X_1  \ar[uu]_{d_0}
\ar[r]|(0.55){s_1} 
&
\ar[l]<+2mm>^{d_1}\ar[l]<-2mm>_{d_2} 
X_2  \ar[uu]_{d_0}
\ar[r]<-2mm>|(0.6){s_1}\ar[r]<+2mm>|(0.6){s_2}  
&
\ar[l]<+4mm>^(0.6){d_1}\ar[l]|(0.6){d_2}\ar[l]<-4mm>_(0.6){d_3}
X_3 \ar[uu]_{d_0}
\ar[r]<-4mm>|(0.6){s_1}\ar[r]|(0.6){s_2}\ar[r]<+4mm>|(0.6){s_3}  
&
\ar[l]<+6mm>^(0.6){d_1}\ar[l]<+2mm>|(0.6){d_2}\ar[l]<-2mm>|(0.6){d_3}\ar[l]<-6mm>_(0.6){d_4}
X_4 \ar[uu]_{d_0}
\ar@{}|\cdots[r]
&
}
$$

  In the present contribution, we shall exploit decalage to relate the fat nerve
  of the Grothendieck construction of a restriction species with its associated
  decomposition space (Proposition~\ref{prop:Dec-of-DRSp} and Corollary~\ref{cor:DecR=NR}), in turn
  important in proving fully faithfulness of the construction of decomposition 
  spaces.  
  
  In a broader perspective, decalage plays an
  important role in the
  theory of decomposition spaces: on one hand,
  many reduction procedures in classical combinatorics can be expressed in terms
  of decalage \cite{GKT:ex}, and on the other hand, the very notion of
  decomposition space can be characterised in terms of decalage, by virtue of
  the following result from \cite[Theorem 4.11]{GKT:DSIAMI-1} (see 
  also \cite{Dyckerhoff-Kapranov:1212.3563}):
\end{blanko}

\begin{thm}\label{thm:decomp-dec-segal}
  $\ds X$ is a decomposition space if and only if $\Dec_\top(\ds X)$ and 
  $\Dec_\bot(\ds X)$ are
  Segal spaces, and the dec maps $d_\top : \Dec_\top(\ds X) \to \ds X$ and $d_\bot :
  \Dec_\bot(\ds X) \to \ds X$ are \culf.
\end{thm}

\begin{blanko}{Right fibrations and left fibrations.}
  (See \cite{GKT:DSIAMI-1}.)
  A functor between simplicial groupoids $f:\ds Y \to \ds X$
  is called a {\em right fibration} if it
  is cartesian on all bottom face maps $d_\bot$.  This implies that it is 
  also cartesian on all generic maps (i.e.~is \culf).
  The terminology is motivated by the case where $\ds Y$ and $\ds X$ 
  are Segal spaces, in which case it corresponds to standard usage in 
  the theory of $\infty$-categories.  If $\ds X$ and $\ds Y$ are fat nerves 
  of categories, then `right fibration' corresponds to groupoid fibration
  in the sense of Street~\cite{Street:fibsinbicats}.

  Similarly, $f$ is called a {\em left fibration} if it
  is cartesian on $d_\top$ (and consequently on all generic maps also).
\end{blanko}

\begin{lemma}\label{lem:DecULF-is-right}
  If $f:\ds Y \to \ds X$ is a \culf functor between decomposition spaces, then
  $\Dec_\bot(f) : \Dec_\bot(\ds Y) \to \Dec_\bot(\ds X)$ is a right fibration of Segal
  spaces.  Similarly, $\Dec_\top(f) : \Dec_\top(\ds Y) \to
  \Dec_\top(\ds X)$ is a left fibration.
\end{lemma}

%%%%%%%%%%%%%%%%%%%%%%%%%%%%%%%%%%%%%%%%%%%%%%%%%%
\section{Two motivating examples and two basic examples}
%%%%%%%%%%%%%%%%%%%%%%%%%%%%%%%%%%%%%%%%%%%%%%%%%%

While many important examples of coalgebras in combinatorics come from
decomposition spaces which are just (fat nerves of) categories, there are
also many examples which do not (directly) come from a category.
(Sometimes, a construction can be made, involving a reduction procedure
\cite{Dur:1986}.)

In this section we first explain the two examples that triggered the
present investigations, and then explain the most basic example from the
two families they belong to.  The first example, Schmitt's  Hopf algebra
of graphs, is an example of a restriction species.  The terminal
restriction species is that of finite sets.  The second example, the
Butcher--Connes--Kreimer Hopf algebra is an example of a new notion we
introduce, directed restriction species, and the terminal such is the
example of finite posets.

\begin{blanko}{The chromatic Hopf algebra (of graphs).}\label{ex:graphs}
  The following Hopf algebra of graphs was first studied by
  Schmitt~\cite{Schmitt:1994},
%   \S 12, 
  and later by 
  Aguiar--Bergeron--Sottile~\cite{Aguiar-Bergeron-Sottile} and 
  Humpert--Martin~\cite{Humpert-Martin}.  For a graph $G$ with vertex set $V$
  (admitting multiple edges and loops), and a subset $U \subset V$, define $G|U$
  to be the graph whose vertex set is $U$, and whose graph structure is induced
  by restriction (that is, the edges of $G|U$ are those edges of $G$ both of
  whose incident vertices belong to $U$).  On the vector space spanned by
  isoclasses of graphs, define a comultiplication by the rule
  $$
  \Delta(G) = \sum_{A+B=V} G|A \tensor G|B .
  $$
  
  This coalgebra is the cardinality of the coalgebra of a {\em
  decomposition space} but not directly of a {\em category}.  Indeed,
  define a simplicial groupoid with $\ds G_1$ the groupoid of graphs, and
  more generally $\ds G_k$ the groupoid of graphs with an ordered partition
  of the vertex set $V$ into $k$ parts (possibly empty), i.e.~a function $V
  \to \un k$ (this is what we shall call a layering~(\ref{layering})).  In
  particular, $\ds G_0$ is the contractible groupoid consisting only of the
  empty graph.  The outer face maps delete the first or last part of the
  graph, and the inner face maps join adjacent parts.  The degeneracy maps
  insert an empty part.  It is clear that this is not a Segal space: a
  graph structure on a given set cannot be reconstructed from knowledge of
  the graph structure of the parts of the set, since chopping up the graph
  and restricting to the parts throws away all information about edges
  going from one part to another.  One can easily check that it {\em is} a
  decomposition space (see \cite{GKT:ex}, where there is also a nice
  picture illustrating the decomposition-space axiom in this case), hence 
  induces a coalgebra.
  Note that disjoint union of graphs makes this into a bialgebra.  With
  grading by the number of vertices, this is a connected graded bialgebra,
  hence a Hopf algebra, clearly precisely Schmitt's chromatic 
  Hopf algebra.
\end{blanko}

\begin{blanko}{Butcher--Connes--Kreimer Hopf algebra.}\label{ex:CK}
  A {\em rooted tree} is a connected and simply-connected graph with a specified
  root vertex; a {\em forest} is a disjoint union of rooted trees.  The
  Butcher--Connes--Kreimer Hopf algebra of rooted trees
  \cite{Connes-Kreimer:9808042} is the free algebra on the set of isoclasses of
  rooted trees, with comultiplication defined by summing over certain admissible
  cuts $c$:
  $$
  \Delta(T) = \sum_{c\in \operatorname{adm.cuts}(T)} P_c \tensor R_c  .
  $$
  An admissible cut $c$ is a splitting of the set of nodes into two subsets,
  such that the second forms a subtree $R_c$ containing the root node (or is the
  empty forest); the first subset, the complement `crown', then forms a
  subforest $P_c$, regarded as a monomial of trees.  Note that compared to the
  arbitrary splitting allowed in Schmitt's Hopf algebra of graphs, the
  admissible cuts are thus required to be compatible with the partial order
  underlying trees and forests.
  
  D\"ur~\cite{Dur:1986} (Ch.IV, \S3) gave an incidence-coalgebra construction of
  the Butcher--Connes--Kreimer coalgebra by starting with the category $\CC$ of
  forests and root-preserving inclusions, generating a coalgebra (in our
  language the incidence coalgebra of the fat nerve of $\CC$, cf.~\cite{GKT:ex}),
  and imposing the
  equivalence relation that identifies two root-preserving forest inclusions if
  their complement crowns are isomorphic forests.  To be precise, this yields
  the opposite of the Butcher--Connes--Kreimer coalgebra, in the sense that the
  factors $P_c$ and $R_c$ are interchanged.  To remedy this, one should just use
  $\CC\op$ instead of $\CC$.
  
  We can obtain the Butcher--Connes--Kreimer coalgebra directly from a
  decomposition space (cf.~\cite{GKT:ex}): let $\ds H_1$ denote the
  groupoid of forests, and let $\ds H_2$ denote the groupoid of
  forests with an admissible cut.  More generally, $\ds H_0$ is
  defined to be a point, and $\ds H_k$ is the groupoid of forests with
  $k-1$ compatible admissible cuts.  These form a simplicial groupoid
  $\ds H$ in which the inner face maps forget a cut, and the outer
  face maps project away either the crown or the bottom layer (the
  part of the forest below the bottom cut).  It is clear that $\ds H$
  is not a Segal space: a tree with a cut cannot be reconstructed from
  its crown and its bottom tree, which is to say that $\ds H_2$ is not
  equivalent to $\ds H_1 \times_{\ds H_0} \ds H_1$.  It is
  straightforward to check that it {\em is} a decomposition space, and
  that its incidence coalgebra is precisely the
  Butcher--Connes--Kreimer coalgebra.
  
  The relationship with D\"ur's construction is this (cf.~\cite{GKT:ex}): the
  `raw' decomposition space $\fatnerve(\CC\op)$ is the decalage of $\ds H$:
  $$
  \Dec_\top \ds H \simeq \fatnerve(\CC\op) .
  $$
  Furthermore, the dec map $\Dec_\top \ds H \to \ds H$, always a \culf
  functor, realises precisely D\"ur's reduction.
    
  As in the graph example, disjoint union makes this 
  coalgebra into a bialgebra.  It is graded by the number of nodes, and since
  the empty forest is the only one without nodes, this bialgebra is connected,
  and hence a Hopf algebra.

  (While the decomposition space $\ds H$ is not a Segal space, it admits
  important variations which {\em are} Segal spaces, namely by replacing the 
  combinatorial trees above by {\em operadic} trees, as explained in 
  \ref{ex:trees}.)
\end{blanko}

\begin{blanko}{Getting decomposition spaces from restriction species and directed restriction species.}
  The graph example is just one in a large family of coalgebras (and bialgebras)
  constructed by Schmitt~\cite{Schmitt:hacs}, namely coalgebras induced by
  restriction species
  (see also \cite{Aguiar-Mahajan}).  We shall show, first of all, that
  restriction species in the sense of Schmitt~\cite{Schmitt:hacs} are examples
  of decomposition spaces, and that they and their associated coalgebras
  exemplify the general construction.  
  The example with trees does not come from a restriction species, but we
  introduce the notion of {\em directed restriction species}, which covers this
  examples and many others, and which also define decomposition spaces.
  
  The next two examples are the basic ones.
\end{blanko}

\begin{blanko}{The binomial Hopf algebra.}
  Define a comultiplication on the vector space spanned by isoclasses of finite 
  sets by
  $$
  \Delta(A) = \sum_{A_1+A_2=A} A_1 \tensor A_2  .
  $$
  Here the sum is over all pairs of subsets of $A$ whose union is $A$.
\end{blanko}

\begin{blanko}{The Hopf algebra of finite posets.}\label{ex:posetcoalg}
  Define a comultiplication on the vector space spanned by isoclasses of finite 
  posets by
  $$
  \Delta(P) = \sum_{c\in \text{cuts}(P)} D_c \tensor U_c   .
  $$
  Here the sum is over all admissible cuts of $P$; an {\em admissible
  cut} $c=(D_c,U_c)$ is by definition a way of writing $P$ as the
  disjoint union of a lower-set $D_c$ and an upper-set $U_c$.
%   (this implies that no element in $D_c$ is greater than any element in $U_c$).
  This coalgebra was studied by 
  Aguiar--Bergeron--Sottile~\cite{Aguiar-Bergeron-Sottile}, who trace its
  origins back to Gessel~\cite{Gessel:84}. See also 
  Figueroa--Gracia-Bond\'ia~\cite{Figueroa-GraciaBondia:0408145}.
\end{blanko}

\section{Simplicial preliminaries}

A key ingredient in our constructions is the beautiful interplay
between the topologist's Delta and the algebraist's Delta.
After setting up the notation, we establish a certain correspondence
between squares in the two categories.

\begin{blanko}{`Topologist's Delta'.}\label{deltadelta}
  The category $\simplexcategory$ is the skeleton of the category 
  of non-empty finite ordered sets and monotone maps. 

\noindent Notation: its objects are
$$
[n]:=\{0,1,\ldots,n\},\qquad n\geq 0 .
$$
The monotone  maps are generated by
\begin{itemize}\item $s^k:[n{+}1]\to[n]$ that repeats the element $k\in [n]$,
\item $d^k:[n]\to[n{+}1]$ that skips the element $k\in [n{+}1]$.
\end{itemize}
Note that $[0]$ is terminal.
\end{blanko}

\begin{blanko}{`Algebraist's Delta'.}
  The category $\unDelta$ is the skeleton of the category of finite ordered
  sets (including the empty set) and monotone maps.

\noindent Notation: its objects are
$$
\un n:=\{1,\ldots,n\},\qquad n\geq 0 .
$$ 
The monotone maps are generated by \begin{itemize}
\item$\un s^k:\un{n{+}1}\to\un n$ that repeats the element $k+1\in\un n$, ($0\le k\le n-1$),
\item $\un d^k:\un n\to\un{n{+}1}$ that skips the element $k+1\in \un{n{+}1}$, ($0\le k\le n$).
\end{itemize}
Note that $\un 1$ is terminal, $\un 0$ is initial, and the only map with 
target $\un 0$ is the identity.
\end{blanko}

There is a full inclusion $\simplexcategory \to \unDelta$ which on
objects sends $[n]=\{0,\dots,n\}$ to $\un{n{+}1}=\{1,\dots,n+1\}$.  On
maps it just does nothing, up to the canonical relabelling of the
elements, $[n]\cong\un{n{+}1}$.  Thus it sends $d^k$ to $\un d^k$ and
$s^k$ to $\un s^k$.  

More important is the following duality, which is 
standard \cite{Joyal:disks}.
\begin{lemma}\label{Joyal-duality}
  There is a canonical isomorphism of categories
  $$
  \Deltagen\op \;\;\cong
  \un \simplexcategory ,
  $$
\end{lemma}
\noindent
\begin{itemize}
		\item $\un n$ corresponds to $[n]$, 
		\item $\un d^k:\un n\to\un{n{+}1}$ corresponds to $s^k:[n{+}1]\to[n]$, 
		\item $\un s^k:\un{n{+}1}\to\un n$ corresponds to the inner coface map 
		$d^{k+1}:[n]\to[n{+}1]$.
	\end{itemize}

	The following graphical representation may be helpful.
	In $\unDelta$, draw the elements in $\un n$ as $n$ dots, and in 
	$\Deltagen$ draw the elements in $[n]$ as $n+1$ walls.
	A map operates as a function on the set of dots when considered a
map in $\unDelta$ while it operates as a function on the walls when considered a map in 
$\Deltagen$.
Here is a picture of a certain map $\un 5 \to \un 4$ in
$\unDelta$ and of the corresponding map $[5] \leftarrow [4]$ in
$\Deltagen$.
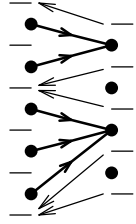
\begin{center}

\begin{texdraw}
  \arrowheadtype t:V
  
    \arrowheadsize l:6 w:3  

  \move (0 0) 
  \bsegment
%   \tdred
  \move (0 -5)
%   \lvec (0 85)
  \move (0 8) \fcir f:0 r:2.5
  \move (0 24) \fcir f:0 r:2.5 
  \move (0 40) \fcir f:0 r:2.5 
  \move (0 56) \fcir f:0 r:2.5 
  \move (0 72) \fcir f:0 r:2.5 
  \move (0 0) \rlvec (-8 0)
  \move (0 16) \rlvec (-8 0)
  \move (0 32) \rlvec (-8 0)
  \move (0 48) \rlvec (-8 0)
  \move (0 64) \rlvec (-8 0)
  \move (0 80) \rlvec (-8 0)
  \esegment
  
  \move (30 0)
    \bsegment
% 	\tdred
    \move (0 3)
%   \lvec (0 77)
  \move (0 16) \fcir f:0 r:2.5 
  \move (0 32) \fcir f:0 r:2.5 
  \move (0 48) \fcir f:0 r:2.5 
  \move (0 64) \fcir f:0 r:2.5 
    \move (0 8) \rlvec (8 0)
    \move (0 24) \rlvec (8 0)
    \move (0 40) \rlvec (8 0)
    \move (0 56) \rlvec (8 0)
    \move (0 72) \rlvec (8 0)

  \esegment

  \move  (27 72) \avec (3 80) 
    \move   (27 55) \avec (3 49)
    \move   (27 41) \avec (3 47)
    \move   (27 22.7) \avec (3 2.3)
  \move   (27 8) \avec (3 0)

        \arrowheadsize l:4 w:3  
% \tdblue
    \linewd 0.9
    \move (0 8) \avec (15 20)\lvec (30 32)
    \move (0 24)\avec (15 28)\lvec (30 32)
    \move (0 40)\avec (15 36)\lvec (30 32)
    \move (0 56)\avec (15 60)\lvec (30 64)
    \move (0 72)\avec (15 68)\lvec (30 64)
\end{texdraw}
\end{center}

\begin{blanko}{Ordinal sum.}
  The ordinal sum monoidal structure $(\unDelta,+,\un 0)$ gives a
  monoidal structure $(\Deltagen,\vee,[0])$, via
  Lemma~\ref{Joyal-duality}.  The free maps $[n]\rat[n']$ in
  $\simplexcategory$ may be expressed uniquely as
  $[n]\rat[a]\vee[n]\vee[b]$.  Any map $[k]\to[n']$ in
  $\simplexcategory$ has a unique factorisation as a generic map
  $f:[k]\genmap[n]$ followed by a free map
  $[n]\rat[a]\vee[n]\vee[b]=[n']$.
\end{blanko}

\begin{blanko}{Pullbacks in $\unDelta$.}
  We shall need the following lemmas, whose proofs are straightforward.
\end{blanko}

\begin{lemma}\label{lem:somepbk1}
  For each $0 \leq k \leq n$, the following square is a pullback in $\unDelta$:
  $$\xymatrix{
     \un{n} \ar[d]_-=\drpullback \ar[r]^-{\un d^k} & \un{n{+}1}\ar[r]^-{\un d^k} & 
     \un{n{+}2} \ar[d]^{\un s^k} \\
     \un{n} \ar[rr]_-{\un{d}^k} && \un{n{+}1}  .
  }$$
\end{lemma}

\begin{lemma}\label{lem:somepbk2}
  For each $0 \leq k \leq n$, the following square is a pullback in $\unDelta$:
  $$\xymatrix{
     \un{n} \ar[d]_=\drpullback \ar[r]^= & \un{n} \ar[d]^{\un d^k} \\
     \un{n} \ar[r]_-{\un{d}^k} & \un{n{+}1}  .
  }$$
\end{lemma}

\begin{lemma}\label{lem:po=pbk}
 For
$0<k<n$
%$\bot<k<\top$
  and all $j$ the following squares are pullbacks
  $$\xymatrix{
     \un n \drpullback  \ar[r]^{\un d^\top}\ar[d]_{\un d^k} & \un{n{+}1} 
     \ar[d]^{\un d^k} \\
     \un {n{+}1} \ar[r]_{\un d^\top} & \un{n{+}2}
  }  \quad
\xymatrix{
     \un n \drpullback  \ar[r]^{\un d^\top}\ar[d]_{\un s^j} & \un{n{+}1} 
     \ar[d]^{\un s^j} \\
     \un {n{-}1} \ar[r]_{\un d^\top} & \un{n}
  }    \quad
\xymatrix{
     \un n \drpullback  \ar[r]^{\un d^\bot}\ar[d]_{\un d^k} & \un{n{+}1} 
     \ar[d]^{\un d^{k{+}1}} \\
     \un {n{+}1} \ar[r]_{\un d^\bot} & \un{n{+}2}
  }
  \quad
\xymatrix{
     \un n \drpullback  \ar[r]^{\un d^\bot}\ar[d]_{\un s^j} & \un{n{+}1} 
     \ar[d]^{\un s^{j+1}} \\
     \un {n{-}1} \ar[r]_{\un d^\bot} & \un{n}  .
  }  
  $$
\end{lemma}

\begin{blanko}{Convex maps.}\label{convexmaps}
  A map $j$ in $\unDelta$ is called {\em convex} and written $j: \un n \rat
  \un n'$ if it is distance-preserving: $j(x+1)= j(x) + 1$, for all
  $x\in\un n$.  (In the subcategory $\simplexcategory\subset \unDelta$ we
  called these `free maps'.  We prefer to use different names since they
  play a different role in the two categories.)  Observe that the convex
  maps are just the canonical inclusions
  $$
  j:\un n\rat\un a+\un n+\un b,
  $$
  and that, for $k>0$, there is a canonical bijection
  $$
  \unDelta_{\mathrm{convex}}(\un k,\un n)
  \;\;\cong\;\;
  \unDelta_{\mathrm{convex}}(\un {k{+}1},\un {n{+}1}).
  $$
  In combination with the full inclusion 
  $\simplexcategory\subset\unDelta$, we get
\end{blanko}

\begin{lemma}\label{lem:canonical-iso-hahaha}
  For $k>0$, there is a canonical isomorphism
  $$
  \simplexcategory_{\mathrm{free}}^{\geq1}
  \;\;\cong\;\;
  \unDelta_{\mathrm{convex}}^{\geq1},\qquad[k]\mapsto\un 
  k\quad(k\geq1).
  $$
\end{lemma}
\noindent Note that this does {\em not} extend to $k\geq0$ 
(since $\un 0$ is initial but $[0]$ is not).

\begin{lemma}\label{convexpbk}
  Convex maps in $\un \simplexcategory$ admit pullback
  along any map: 
  given the solid cospan consisting of $g$ and $i$,
  with $i$ convex,
  $$\xymatrix{
     \un n' \ar[d]_g & \un n \dlpullback \ar@{-->}[d]^f \ar@{ >-->}[l]_{j} \\
     \un k'  & \ar@{ >->}[l]^i \un k\,   ,
  }$$
the pullback exists and $j$ is again convex.  
\end{lemma}

\begin{lemma}\label{lem:free-gen-Delta}
  For $k>0$, there is a bijection between the set of pullback squares
  along convex maps in $\un \simplexcategory$ and the set of
  commutative squares of generic against free maps in
  $\simplexcategory$
  $$
  \left\{\vcenter{\xymatrix{
     \un n' \ar[d]&\ar@{ >->}[l]  \dlpullback \un n \ar[d] \\
  \un    k' &\ar@{ >->}[l] \un  k}} \quad \text{in } 
  \;\unDelta\right\}
  \qquad = \qquad
  \left\{\vcenter{\xymatrix{
     [n']  & \ar@{ >->}[l] [n] \\
      [k'] \ar@{->|}[u] & \ar@{ >->}[l] [k] \ar@{->|}[u]
  }}
\quad \text{ in } \;\simplexcategory \right\}
_{\raisebox{0.6ex}[0ex][0ex]{.}}
  $$
  The bijection is given by Lemma~\ref{Joyal-duality} on the vertical maps,
  and by Lemma~\ref{lem:canonical-iso-hahaha} on the bottom horizontal map.
\end{lemma}
\noindent
In the case $k=0$, we necessarily have $n=0$ and $n'=k'$, but 
there is not even a bijection on the bottom arrows in this case.  
\begin{proof}
  The bijection is the composite of the three bijections
  $$
  \left\{\vcenter{\xymatrix{
     \un n' \ar[d]&\ar@{ >->}[l] \dlpullback \un n \ar[d] \\
  \un    k' &\ar@{ >->}[l] \un  k}} \right\}
   =   
   \left\{\vcenter{\xymatrix{
      \un n' \ar[d] & \\
  \un k' &\ar@{ >->}[l] \un  k}} \right\}
   = 
  \left\{\vcenter{\xymatrix{
      [n'] &\\
[k'] \ar@{->|}[u] & \ar@{ >->}[l]  [k] 
}} \right\} = 
  \left\{\vcenter{\xymatrix{
     [n']  & \ar@{ >->}[l] [n] \\
      [k'] \ar@{->|}[u] & \ar@{ >->}[l] [k] \ar@{->|}[u]
}} \right\}
  $$
   where the first bijection is by existence of pullbacks along convex maps 
   (Lemma \ref{convexpbk}), the 
   second is by Lemmas~\ref{Joyal-duality} and~\ref{lem:canonical-iso-hahaha}
   (here we use that $k>0$), and the
   third is by unique generic--free factorisation of the composite
   $[k]\rat[k']\genmap[n']$. It can be checked that the bijection between 
   the right-hand arrows is again that of Lemma~\ref{Joyal-duality}.
   In fact, the bijection is
    $$
  \left\{\vcenter{\xymatrix{
     \un a_1+\un n+\un a_2 \ar[d]_{\un g_1+\un f+\un g_2}&\ar@{ >->}[l] \dlpullback \un n \ar[d]^{\un f} \\
  \un b_1+\un    k+\un b_2 &\ar@{ >->}[l] \un  k}} \right\}
=
 \left\{\vcenter{\xymatrix{
     [a_1]\vee[n]\vee[a_2]  & \ar@{ >->}[l] [n] \\
      [b_1]\vee[k]\vee[b_2] \ar@{->|}[u]^{g_1\vee f\vee g_2} & \ar@{ >->}[l] [k] \ar@{->|}[u]_f
    }} \right\}
_{\raisebox{0.6ex}[0ex][0ex]{.}}
  $$
\end{proof}

\begin{blanko}{Identity-extension squares.}\label{iesq}
  A square in $\un \simplexcategory$ is called   is called an {\em identity-extension square (iesq)} if is it of the form
    \begin{equation}\label{iesq-unD}\vcenter{\xymatrix{
    \un a+\un n+\un b  \ar[d]_{\id_a+f+\id_b} & \un n \ar@{ >->}[l]_-j \ar[d]^f &   \\
     \un a+\un k+\un b  & \un k \,, \ar@{ >->}[l]^-i & 
  }}\end{equation}
where $i$ and $j$ are convex.   Note that 
  an iesq 
  is both a pullback and a pushout.
\end{blanko}

\begin{lemma}\label{lem:iesq-genfree}
	Under the correspondence of Lemma~\ref{lem:free-gen-Delta}, 
	identity-extension squares in $\unDelta$ correspond to
	generic-free pushouts in $\simplexcategory$.
\end{lemma}

%%%%%%%%%%%%%%%%%%%%%%%%%%%%%%%%%%%%%%%%%%%%%%%%%%
\section{The decomposition space $\ds I$ of layered finite sets}
%%%%%%%%%%%%%%%%%%%%%%%%%%%%%%%%%%%%%%%%%%%%%%%%%%
\label{sec:B}

Let $\I$ be the category of finite sets and injections.  We define and
study the monoidal decomposition space $\ds I$ of {\em layered finite
sets}: finite sets with an ordered partition into any number of
possibly empty layers.  It is equivalent to the monoidal nerve of the
monoidal groupoid of finite sets and bijections, but the layering viewpoint
will generalise nicely to the directed case (\S\ref{sec:C}).

\begin{blanko}{The groupoid of $n$-layered finite sets.}\label{layering}
  An $n$-{\em layering}, or just a layering, of a finite set $A$ is a function $p:A\to\un n$.
  We refer to the fibres $A_i=p^{-1}(i)$, $i\in\un n$, as {\em layers}. Layers may be empty.
  We consider the groupoid $\ds I_n:=\I^{\iso}_{/\un n}$ of all  $n$-layerings of finite sets, whose arrows are commutative triangles,
  $$
  \xymatrix@C=2ex{A\drto\rrto^\simeq&&A'\dlto\\&\un n.}
  $$
\end{blanko}

\begin{blanko}{The simplicial groupoid of layered finite sets.}\label{simp-layering}
  We now assemble the groupoids of layered finite sets into a simplicial
  groupoid.  For a generic map $g:[n]\to [m]$ of $\simplexcategory$,
  consider the map $g\upperstar:\I^\iso_{/m} \to \I^\iso_{/n}$ given by
  postcomposition with the corresponding map $\un g:\un m\to\un n$ of
  $\unDelta$ under the correspondence of Lemma~\ref{Joyal-duality},
  $$
  g\upperstar:=\un g\lowershriek:  \I^\iso_{/\un m}\to  \I^\iso_{/\un n},\quad (A{\to}\un m )
  \mapsto (A{\to}\un m {\stackrel{\un g}\to} \un n).
  $$
  To define the outer face maps $d_\bot,d_\top : \I^\iso_{/\un k} \to \I^\iso_{/\un
  {k\!-\!1}}$, we take $A{\to}\un k$ to the pullbacks
  $$\xymatrix{
    A' \drpullback \ar@{^{(}->}[r] \ar[d]_{d_\bot (a):={\un d^\bot}\upperstar(a)} & A \ar[d]^a \\
  \un{k{-}1} \ar[r]_-{\un d^\bot} & \un k ,
}\qquad\qquad \xymatrix{
  A' \drpullback \ar@{^{(}->}[r] \ar[d]_{d_\top (a):={\un d^\top}\upperstar(a)} & A \ar[d]^a \\
  \un{k{-}1} \ar[r]_-{\un d^\top} & \un k ,}
$$
  projecting away the first or the last layer.  We make the specific choice
  that the pullbacks are given by subsets; this will ensure that the
  simplicial object we are defining is strict.  More abstractly, for a free
  map $f:[n]\to [m]$ of $\simplexcategory$, the map $ f\upperstar:
  \I^\iso_{/\un m} \to \I^\iso_{/\un n}$ is defined by pullback along the
  corresponding convex map $\un f:\un n\to\un m$ in {$\unDelta$}, given for
  $n\geq1$ by the correspondence of Lemma~\ref{lem:canonical-iso-hahaha}
  between free maps in $\simplexcategory$ and convex maps in $\unDelta$.
  Note that all maps $[0] \to [n]$ correspond to the unique map
  $\un 0 \to \un n$.
\end{blanko}  

\begin{prop}\label{B/k}
  The groupoids $\ds I_n$ and the maps $g\upperstar,f\upperstar$ above form
  a 
%STRICT%   strict 
  simplicial groupoid $\ds I$, which is 
%STRICT%   a strict category object and 
  a Segal space, and hence a decomposition space.
\end{prop}

\begin{proof}
  The generic-generic simplicial identities are already known to hold
  by construction, because they correspond under $\Deltagen\op \simeq
  \unDelta$ to identities in $\unDelta$.
  
  We need to check the following nine simplicial identities
  involving outer face maps:
  \begin{align*}
&&\makebox[0em][l]{\hspace*{-43mm}$d_\top \circ d_\bot = d_\bot \circ d_\top$}\\
	d_\bot \circ d_\bot  &= d_\bot \circ d_1&	d_\top \circ d_\top  &= d_\top \circ d_{\top-1}\\
	d_\bot \circ s_\bot &= \id &d_\top \circ s_\top &= \id \\
	s_k \circ d_\bot &= d_\bot \circ s_{k+1}&s_k \circ d_\top &= d_\top \circ s_{k}\\
	d_k \circ d_\bot &= d_\bot \circ d_{k+1}&d_k \circ d_\top &= d_\top 
	\circ d_{k} .
  \end{align*}
  These relations, according to the definitions we have given of outer 
  face maps in $\ds I$, translate into the following relations between 
  pullback (upperstar) and postcomposition (lowershriek) operations, using 
  the dictionary compiled in Lemma~\ref{Joyal-duality}.
  \begin{align*}
&&\makebox[0em][l]{\hspace*{-55mm}$\un d^\top{}\upperstar  \circ \un d^\bot{}\upperstar = \un d^\bot{}\upperstar \circ \un d^\top{}\upperstar$} \\ 
    {\id\lowershriek} \circ {\un d^\bot}\upperstar \circ \un d^\bot{}\upperstar &= {\un d^\bot}\upperstar \circ {\un s^\bot} \lowershriek&
    {\id\lowershriek} \circ {\un d^\top}\upperstar \circ \un d^\top{}\upperstar &= {\un d^\top}\upperstar \circ {\un s^\top} \lowershriek\\
    {\un d^\bot}\upperstar \circ {\un d^\bot}\lowershriek &= {\id\lowershriek}\circ \id\upperstar &
    {\un d^\top}\upperstar \circ {\un d^\top}\lowershriek &= {\id\lowershriek}\circ \id\upperstar \\
    {\un d^k} \lowershriek \circ \un d^\bot{}\upperstar&= {\un d^\bot}\upperstar  \circ {\un d^{k+1}}\lowershriek &
    {\un d^k} \lowershriek \circ \un d^\top{}\upperstar&= {\un d^\top}\upperstar  \circ {\un d^{k}}\lowershriek \\
    {\un s^{k-1}}\lowershriek  \circ \un d^\bot{}\upperstar &= {\un d^\bot}\upperstar \circ{\un s^k}\lowershriek &
    {\un s^{k-1}}\lowershriek  \circ \un d^\top{}\upperstar &= {\un 
	d^\top}\upperstar \circ{\un s^{k-1}}\lowershriek .
  \end{align*}
  The first of these is induced from a commutative square in $\unDelta$.
  The other eight hold by Beck--Chevalley, since the squares in $\unDelta$ are 
  pullbacks by Lemmas~\ref{lem:somepbk1}--\ref{lem:po=pbk}.
  
%   \tiny
  
% Short proof:  
%   There are five cases, each corresponding to a simplicial identity involving
%   outer face maps.  Each of these identities then correspond to pullback along 
%   some maps in $\unDelta$ combined with postcomposing with other.
%   In each case the identity follows from Beck--Chevalley since these squares
%   are pullbacks in $\unDelta$.
  
%   \normalsize
  
  The simplicial identities can be arranged to
  hold on the nose: the only subtlety is the pullback construction
  involved in defining the outer face maps, but these pullbacks
  can all be chosen to be always actual subset inclusions.

  Finally, since $\I^\iso_{/\un{0}}\simeq 1$, the Segal condition
  says (for each $m,n$) the projection map
  $
  \I^\iso_{/\un{m{+}n}}  \to
  \I^\iso_{/\un{m}}  \times \I^\iso_{/\un{n}}
  $
  must be an equivalence.  But this is clear, since an inverse is given by
  sending $(A{\to}\un m,B{\to}\un n)$ to $A{+}B\to \un{m{+}n}$. 
%STRICT%   If we assume $\I$ to be skeletal (as we do),
%STRICT%   then this equivalence is an isomorphism, so that also the strict
%STRICT%   Segal condition holds.
\end{proof}

\begin{lemma}\label{lem:BB-local}
  The decomposition space $\ds I$ is complete, locally finite, locally discrete,
  and of locally finite length.
\end{lemma}

\begin{proof}
  The checks are straightforward verifications.  (Some 
  indications can be found in the similar Lemma~\ref{lem:Cfinite}.)
\end{proof}

\begin{prop}\label{prop:dec-bb-is-ii}
  The lower Dec of $\ds I$ is
  naturally equivalent to $\fatnerve\I$, the fat nerve of finite sets and injections.
  This equivalence identifies a map
  $A \to \un k$ with the string of $k-1$ injections
  $$
  A_1 \into A_1 + A_2 \into \dots \into A_1 + \cdots + A_{k-1} \into A_1 + 
  \cdots + A_k  .
  $$
  
  (Similarly, the upper dec $\Dec_\top(\ds I)$ is naturally equivalent to 
  $\fatnerve\I\op$.)
\end{prop}

We refer to \cite{GKT:ex} for a proof.  The fat nerve of finite sets
and injections is the approach of D\"ur~\cite{Dur:1986} to the
binomial coalgebra, as explained in~\cite{GKT:ex}. 

\begin{lemma}\label{lem:BB-monoidal}
  $\ds I$ is a monoidal decomposition space under disjoint union.
\end{lemma}

\begin{proof}
  As the proof of Lemma~\ref{lem:CC-monoidal}, but changing $\C$ to $\I$ and $\ds C$ to $\ds I$ everywhere.
\end{proof}

%%%%%%%%%%%%%%%%%%%%%%%%%%%%%%%%%%%%%%%%%%%%%%%%%%
\section{Restriction species}
%%%%%%%%%%%%%%%%%%%%%%%%%%%%%%%%%%%%%%%%%%%%%%%%%%

\label{sec:RSp}

\begin{blanko}{Schmitt's restriction species.}
  Recall that $\I$ denotes the category of finite sets and injections.
  Schmitt~\cite{Schmitt:hacs} defines restriction species to be
  presheaves on $\I$,
  \begin{eqnarray*}
    R:\I\op & \longrightarrow & \Set  \\
    A & \longmapsto & R[A] .
  \end{eqnarray*}
  An element $X$ of $R[A]$ is called an {\em $R$-structure} on the set $A$.
  Compared to a classical species~\cite{JoyalMR633783}, a restriction species
  $R$ is thus functorial not only in bijections but also in injections, meaning
  that an $R$-structure on a set $A$ induces also such a structure on every
  subset $B \subset A$ (denoted with a restriction bar):
  \begin{eqnarray*}
    R[A] & \longrightarrow & R[B]  \\
    X & \longmapsto & X|B .
  \end{eqnarray*}

  A morphism of restriction species is just a natural transformation
  $R\Rightarrow R'$ of functors $\I\op\to\Set$, i.e.~for each finite set $A$ a
  map $R[A]\to R'[A]$, natural in $A$. 
\end{blanko}

\begin{blanko}{Schmitt construction.}\label{Schmitt-comult}
  The Schmitt construction~\cite{Schmitt:hacs} associates to a restriction
  species $R: \I\op\to\Set$ a (cocommutative) coalgebra structure on the
  vector space spanned by the isoclasses of $R$-structures: the
  comultiplication is
  $$
  \Delta(X)=
  \sum_{A_1+A_2=A} X|A_1 \tensor X|A_2 , \qquad X \in R[A] ,
  $$
  and the counit sends $X\in R[\varnothing]$ to $1$ and other structures to $0$.
  
  Since the summation in the
  comultiplication formula only involves the underlying sets, it is readily seen
  that a morphism of restriction species induces a coalgebra homomorphism.

  A great many (cocommutative) combinatorial coalgebras can be realised by
  the Schmitt construction (see \cite{Schmitt:hacs} and also
  \cite{Aguiar-Mahajan}).  For example, graphs (\ref{ex:graphs}), matroids,
  simplicial complexes, posets, categories, etc., form restriction species
  and hence coalgebras.  In many cases, disjoint union furthermore defines
  an algebra structure, and altogether a bialgebra.  Finally, in most
  cases, $R[\varnothing]$ is singleton.  This implies that the bialgebra is
  connected and hence a Hopf algebra.  Schmitt actually includes this
  condition in his definition of restriction species.  In the present work,
  we shall {\em not} assume $R[\varnothing]$ singleton.
\end{blanko}

\begin{blanko}{Groupoid-valued species.}\label{Gr-sp}
  In line with our general philosophy, we shall work with groupoids
  rather than sets, aspiring to a native treatment of symmetries.
  Groupoid-valued species were first advocated by Baez and
  Dolan~\cite{Baez-Dolan:finset-feynman} (who called them {\em stuff
  types}, as opposed to {\em structure types}, their translation of Joyal's
  {\em esp\`eces de structures}~\cite{JoyalMR633783}), for the sake of 
  dealing with symmetries of Feynman
  diagrams.  They showed also that over groupoids (but not over sets),
  the generating function of a species is the homotopy cardinality of
  its associated analytic functor.  Furthermore, over groupoids,
  analytic functors are polynomial~\cite{Kock:MFPS28}, meaning that
  they are given by pullback functors and their adjoints.  Since the 
  decomposition-space machinery is based on homotopy pullbacks and 
  homotopy cardinality, we may as well consider groupoid-valued 
  species, which we do from now on. 
  
  For the sake of taking cardinality, it is furthermore natural to require
  the groupoid values to be locally finite.  This means that every object
  has finite automorphism group.  This is usually the case of combinatorial
  objects.  In particular, every set (finite or not) is locally finite.  So
  a classical species is always locally finite.
\end{blanko}

\begin{blanko}{Restriction species.}
  A {\em 
%STRICT%   (strict) 
  restriction species}
  is a groupoid-valued presheaf on $\I$,
  \begin{eqnarray*}
    R:\I\op & \longrightarrow & \Grpd  \\
    A & \longmapsto & R[A] .
  \end{eqnarray*}
  A morphism of restriction species is a
%STRICT%   (strict) 
  natural transformation.
  We actually allow pseudo-functors and pseudo-natural 
  transformations, but make some remarks on the strict case in 
  \S\ref{sec:strict}.
  This defines the category $\kat{RSp}$ of restriction species.
  
A restriction species corresponds, by the Grothendieck construction, to a 
%STRICT% split
right fibration
(i.e.~a
%STRICT% split 
cartesian fibration with groupoid fibres)
$$
\R \to \I .
$$
Here $\R$ is the category of elements of $R$, whose objects are
$R$-structures and whose arrows are structure-preserving injections.  More
precisely, an object is a pair $(A,X)$ where $A$ is a finite set and
$X\in R[A]$, and a morphism
$(A',X')\to(A,X)$ is an injection $A'\to A$ in $\I$ and an arrow
$X' \isopil X|A'$ in the groupoid $R[A']$.
The category of restriction species
is canonically equivalent to the categories of groupoid-valued 
presheaves on $\I$, and of right fibrations over $\I$:
$$
\kat{RSp}\simeq\Grpd^{\I\op}\simeq\kat{RFib}_{/\I} .    
$$
It is sometimes more informative to describe a restriction species by describing
the right fibration $\R\to\I$ rather than describing the functor 
$R:\I\op\to\Grpd$, because 
the description of the category $\R$ already has the specifics about the 
restrictions, encoded in the arrows of the category.  We shall see
this in the examples.
\end{blanko}

%STRICT% \begin{blanko}{Remark.}
%STRICT%   We have defined restriction species to be strict functors $R:\I\op\to\Grpd$,
%STRICT%   corresponding to split right fibrations $\R\to\I$, and required morphisms 
%STRICT%   to be strictly natural transformations, which for the corresponding fibrations 
%STRICT%   means that chosen cartesian arrows are mapped to chosen cartesian arrows.
%STRICT%   These strictness assumptions will be invoked later when we define simplicial
%STRICT%   groupoids and functors, as they will ensure that these are strict.
%STRICT%   
%STRICT%   In the end, however, we are concerned with homotopy invariant notions,
%STRICT%   and we might as well allow simplicial groupoids $X: \simplexcategory\op\to\Grpd$ 
%STRICT%   to be pseudo-functors, and allow their morphisms to be pseudo-natural 
%STRICT%   transformations.  If one does not mind these weak invariant notions, it is not
%STRICT%   necessary either to require strict restriction species or strict morphisms
%STRICT%   between them.  
%STRICT% %   For the sake of emphasising the essential ideas, we have chosen
%STRICT% %   to stick with the strict notions.
%STRICT% \end{blanko}

\begin{blanko}{Examples of restriction species.}\label{ex:RSp}
  (See~\cite{Schmitt:hacs} for these and more examples.)
  
  (1) {\em Graphs.} The species of finite graphs is a restriction
  species, cf.~Example~\ref{ex:graphs}.  It is fruitful to look at it also
  as a right fibration $\G\to \I$: the category $\G$ is then the category
  whose objects are finite graphs, and whose morphisms are full graph
  inclusions.  Full means that if two vertices $x$ and $y$ are in the
  subgraph then all edges between $x$ and $y$ must also be included.
  (Allowing non-full inclusions, such as \begin{texdraw} \move (0 0)\onedot
  \move (10 0) \onedot\end{texdraw} $\into$ \begin{texdraw} \move (0
  0)\onedot \lvec (10 0) \onedot\end{texdraw}, would prevent $\G\to\I$ from
  being a right fibration.)
  
  (2) {\em Matroids.} (See Oxley~\cite{Oxley} for definitions.)
  The species of matroids is a restriction species \cite{Schmitt:hacs}.
  Many important classes of matroids are stable under restriction and 
  are therefore also restriction species.  For example,
  transversal matroids,
  % See Joseph Bonin
  representable matroids,
  % obtained from a matrix over a field by taking E the set of columns
  % and \mathcal{I} consisting of independent sets of columns.
  % It is clead that one can restrict to subsets of columns.
  % Binary matroid and ternary matroid just mean representable over
  % F_2 and F_3
  regular matroids, graphic matroids, bond matroids, planar matroids, and so on.
 
  (3) {\em Posets.} \label{ex:posets} The species of posets is a
  restriction species.  The corresponding right fibration is $\P \to \I$,
  where $\P$ is the category of finite posets and full poset inclusions
  $F\into P$.  `Full' means that for two elements $x, y$ in $F$ we have
  $x\leq_F y$ if and only if $x\leq_P y$.
  
  In \S\ref{sec:DRSp} we shall introduce directed restriction species, based on
  a different category of posets, namely the category $\C$ of finite posets and
  convex maps.  The forgetful functor $\C\to\I$ is {\em not} a right 
  fibration: there is no convex lift of the set inclusion $\{0,2\} \into \{0,1,2\}$
  to the linear order $\{0\leq 1 \leq 2\}$.

  (4) {\em Categories.} The species of finite categories assigns
  to a finite set the groupoid of all finite-category structures
  on that set of objects.  In this case the right fibration is $\F\to\I$,
  where $\F$ is the category of finite categories and full subcategory
  inclusions (or more precisely, injective-on-objects fully faithful
  functors).  The underlying-set functor $\F\to\I$ is a right fibration
  because clearly any subset of the object set of a category determines
  uniquely a full subcategory.
  
  Note: in the examples of graphs and categories we stress the word
  `finite': if we allowed an infinite number of edges/arrows between
  two elements, an infinite automorphism group would result, violating the
  local finiteness assumption made in \ref{Gr-sp}.
\end{blanko}

\begin{blanko}{Slices of examples.}
  Recall that for any object $x$ in a category $\CC$, the domain projection
  $\CC_{/x} \to \CC$ is a right fibration.  In particular, if $\R\to\I$ is a 
  restriction species, for any $R$-structure $X$, the slice category 
  $\R_{/X}$ is again a restriction species.  It is the restriction species 
  of $R$-substructures of $X$.  See Bergner et 
  al.~\cite{Bergner-et.al:1609.02853}
  for examples of slices of the decomposition space of graphs.
  The fact that slicing a restriction species
  produces again restriction species reflects the local nature of 
  coalgebras: every element in a coalgebra generates a coalgebra.
\end{blanko}

\begin{blanko}{Restriction species as decomposition spaces over $\ds I$.}
  From a restriction species $R$, or a 
%STRICT%   split
  right fibration $\R\to\I$, we shall construct a simplicial groupoid $\ds
  R$ of layered $R$-structures, together with a \culf functor $\ds R\to\ds
  I$.

  As in \S\ref{sec:B}, the subtlety is that the obvious functoriality is in
  $\unDelta \simeq \Deltagen\op$, not in all of 
  $\simplexcategory\op$.
  Consider first the functor $\I^{\iso}_{/_-}:\unDelta\to \Grpd$ and form
  the pullbacks
  $$
  \ds R_k = \I^{\iso}_{/\un k} \times_{\I^{\iso}_{/\un 1}} \R^\iso 
  $$
  along the functor $\R^\iso\to\I^\iso=\I^{\iso}_{/\un 1}$.  Thus $\ds R_k$
  is the groupoid of $R$-structures with a $k$-layering of the underlying
  sets.
This defines a diagram of shape $\unDelta=\Deltagen\op$:
$$
R_{\operatorname{gen}}:\unDelta\to\Grpd  .
$$
The pullback construction also shows that forgetting the $R$-structure and 
retaining only the layering of the underlying set provides a cartesian 
natural transformation (of $\Deltagen\op$-diagrams)
$$
R_{\operatorname{gen}}\to I_{\operatorname{gen}}.$$ 

So far the construction works for any species, not necessarily restriction 
species.
To define also the free maps (i.e.~outer face maps)
  we need the restriction structure on $R$, which allows us to lift the 
  outer face maps we constructed for $\ds I$.  Recall that  the outer face map
  $d_\bot : \I^{\iso}_{/\un k} \to \I^{\iso}_{/\un{k{-}1}}$ is defined by 
  sending $A {\to} \un k$ to
  the pullback
  $$\xymatrix{
  A' \ar[d]\drpullback \ar[r]^\subset & A \ar[d]\\
  \underline{k\!-\!1} \ar[r]_-{\un d_\bot} & \underline{k} .
  }$$
  Since $A' \into A$ is an injection, we can use functoriality of $R$
  (the fact that $R$ is a restriction species) to get also the face
  map for $\ds R_k$: for example,
  $$
  d_\bot: 
  \ds R_k \to \ds R_{k-1}
  $$
  is defined as
  $$
  \big(A{\to}\un k ,\,X\big) \mapsto 
  \big(\un d_\bot\upperstar\!A{\to}\un{k{-}1},\,X|\un d_\bot\upperstar\!A\big)   .
  $$
   
  We see that the point is to be covariantly functorial in all maps in
  $\unDelta$ and to be contravariantly functorial in convex 
  maps.  To establish the simplicial identities is to exhibit a 
  certain compatibility between these two
  functorialities.  These conditions are precisely condensed in the notion of
  sesquicartesian fibration which we introduce in \S\ref{sec:sesq} below.
\end{blanko}

\begin{theorem}\label{RisDS}
  Given a restriction species $R$, the above construction defines a 
  simplicial groupoid  $\ds R$, which is a decomposition space.  Furthermore, a morphism of restriction species $R'\to R$
  induces a \culf functor $\ds R' \to \ds R$. 
  These assignments define
  a functor from the category of restriction species to that of decomposition spaces and \culf functors.
\end{theorem}

\begin{proof}
  The simplicial identities in $\ds R$ can be checked by hand, arguing along the lines of
  the proof of Proposition~\ref{B/k}.  (Later we will give a more elegant
   proof using the machinery introduced in
  Sections~\ref{sec:nabla}--\ref{sec:RSp->iesq} and there will be no need for ad hoc
  arguments).  Since by construction the simplicial groupoid $\ds R$ is \culf
  over a decomposition space $\ds I$, it is itself a decomposition space (by
  Lemma~\ref{lem:CULF}).

  A morphism $f:R' \to R$ amounts to a morphism of right fibrations
  $$\xymatrix@C=2ex{
  \R' \ar[rd]\ar[rr] && \R \ar[ld] \\
  & \I &}
  $$
  inducing simplicial maps
  $$\xymatrix@C=2ex{
  \ds R' \ar[rd]\ar[rr]^{\ds f} && \ds R \ar[ld] \\
  & \ds I   . &}
  $$
  Indeed, at level $n$, the morphism of groupoids $\R'{}_{/\un n} \to 
  \R_{/\un n}$ is induced from $\R' \to \R$, since the layering only 
  affects the underlying set which does not change.  Finally, $\ds f$
  is \culf since the projection maps to $\ds I$ are.
\end{proof}

\begin{blanko}{Decalage.}\label{DecR=NR}
  The decomposition space $\ds R$ constructed from the restriction species
  $\R$ can be seen as an `un-decking': we have
  $$
  \Dec_\bot \ds R  \simeq \fatnerve \R , \qquad
  \Dec_\top \ds R \simeq \fatnerve \R\op .
  $$
  We postpone the proof until \ref{cor:DecR=NR}.
\end{blanko}
  
\begin{lemma}\label{lem:R1}
    The groupoid $\ds R_1 = \R^{\iso}$ is locally finite.
\end{lemma}
\begin{proof}
    For each $n\in \I^{\iso}$ we have a fibre sequence (homotopy pullback)
    $$\xymatrix{
    R[n] \drpullback \ar[r] \ar[d] & \R^{\iso} \ar[d] \\
    1 \ar[r]_-{\name n} & \I^{\iso}  .
    }$$
    Since $\I^{\iso}$ is locally finite, and since $R[n]$ is locally finite
    by our standing assumption, also $\R^{\iso}$ is locally finite.
\end{proof}

\begin{prop}\label{prop:Rlocallydiscrete}
  The decomposition space $\ds R$ is complete, locally finite, locally
  discrete, and locally of finite length.
\end{prop}

\begin{proof}
  $\ds R_1$ is locally finite by Lemma~\ref{lem:R1}.  The
  remaining finiteness properties and the discreteness property follow from
  Lemmas~\ref{lem:culf-local} and~\ref{lem:BB-local} since $\ds R$ is
  \culf over $\ds I$.
\end{proof}

\begin{blanko}{Coalgebras.}
  (See~\cite{GKT:DSIAMI-1} and \cite{GKT:DSIAMI-2}.)  To any
  decomposition space $\ds X$, there is associated a coalgebra at the
  objective level, namely a comultiplication functor $\Delta:\Grpd_{/\ds X_1}
  \to \Grpd_{/\ds X_1} \tensor \Grpd_{/\ds X_1}$ and a counit functor
  $\varepsilon : \Grpd_{/\ds X_1} \to \Grpd$.    
  Similarly a \culf functor
  $\ds X' \to \ds X$ induces a coalgebra homomorphism, i.e.~a linear
  functor $\Grpd_{/\ds X'_1} \to \Grpd_{/\ds X_1}$ compatible with the
  coalgebra structures.  If the decomposition spaces are locally
  finite, one can take homotopy cardinality to obtain coalgebras over
  $\Q$ and coalgebra homomorphisms in the classical sense.  It is
  outside the scope of the present paper to go into details,
  and we only sketch the proof of the following proposition which is the
  motivation for channelling the Schmitt construction through
  decomposition spaces.  
\end{blanko}

\begin{prop}
  For $R$ a restriction species, the Schmitt coalgebra of $R$ is the
  homotopy cardinality of the incidence coalgebra of the 
  associated decomposition space $\ds R$.  For a morphism of 
  restriction species $R'\to R$, Schmitt's coalgebra homomorphism is
  the cardinality of the associated \culf functor $\ds R' \to \ds R$.
\end{prop}

\begin{proof}
  (Sketch).  At the objective level, the comultiplication is
  given by pullback along $d_1: \ds R_2 \to \ds R_1$, followed by composing
  with $(d_2,d_0)$.  For a given $R$-structure $X$, viewed as a morphism
  $\name{X}:1 \to \ds R_1$, the pullback is the $d_1$-fibre over $X$, that
  is the groupoid $(\ds R_2)_X$ of all $R$-structures with a $2$-layering
  such that the union of the two layers is $X$.  This is a groupoid over
  $\ds R_1 \times \ds R_1$ by composing with $(d_2,d_0)$, which amounts to
  returning the restriction of $X$ to each of the two layers.  To recover
  the formula in \ref{Schmitt-comult}, it remains to take homotopy
  cardinality of this groupoid, relative to $\ds R_1\times \ds R_1$.  This
  is meaningful since $\ds R$ is locally finite by
  Proposition~\ref{prop:Rlocallydiscrete}.  There are general formulae for
  this in \cite{GKT:ex}, but in the present case it is straightforward:
  since $\ds R$ is locally discrete by Proposition~\ref{prop:Rlocallydiscrete},
  the groupoid $(\ds R_2)_X$ is discrete, and hence homotopy cardinality
  amounts to counting isomorphism classes, yielding Schmitt's formula in
  \ref{Schmitt-comult}.  The statement about morphisms does not present 
  further difficulties.
\end{proof}

\begin{blanko}{Monoidal restriction species.}\label{monoidalrestr}
  We introduce the notion of monoidal restriction species.  The idea is simply
  that many restriction species are `closed under disjoint union', in a way
  compatible with restrictions.  This compatibility with restrictions ensures
  that the resulting algebra structure is compatible with the coalgebra
  structure to result altogether in a bialgebra.  This bialgebra is always
  graded (by the number of elements in the underlying set), and most often
  connected (this happens when there is only one possible structure on the empty
  set), and hence a Hopf algebra.  Schmitt~\cite{Schmitt:hacs} arrives at Hopf
  algebras through a notion of {\em coherent exponential restriction species}.
  Our notion is a bit more general, and conceptually simpler.
  
  The category $\I$ has a symmetric monoidal structure given by disjoint
  union, as already exploited to make $\ds I$ a monoidal
  decomposition space (Lemma~\ref{lem:BB-monoidal}).  We define a {\em monoidal
  restriction species} to be a right fibration $\R \to \I$ for which the
  total space $\R$ has a monoidal structure $\sqcup$ and the projection to
  $\I$ is strong monoidal.
  
  If $X_1$ is an $R$-structure with underlying set $S_1$, and $X_2$ is an
  $R$-structure with underlying set $S_2$, and if $K_1\subset S_1$ and $K_2
  \subset S_2$ are subsets (or injective maps), then there is a canonical
  isomorphism
  $$
  (X_1\sqcup X_2)\mid (K_1 + K_2) \simeq (X_1\mid K_1) \sqcup 
  (X_2\mid K_2) .
  $$
  This follows from unique comparison between cartesian lifts and the fact that
  the projection is strong monoidal.  This isomorphism expresses the desired
  compatibility between the monoidal structure and restrictions.
  
  A {\em morphism of monoidal restriction species} is a strong monoidal functor
  which is also a morphism of right fibrations.
\end{blanko}

\begin{prop}\label{prop:monoidalrestr}
  The functor of Theorem \ref{RisDS} extends to a functor from the category 
  of monoidal restriction species and their morphisms to that of
  monoidal decomposition spaces and \culf monoidal functors.
\end{prop}
  
\begin{proof}
  If $\R$ is a monoidal restriction species, then the associated decomposition
  space $\ds R$ is monoidal: in degree $n$, this is simply given by the monoidal
  structure $\sqcup :\R_{/\un n} \times \R_{/\un n} \to \R_{/\un n}$.  This is
  well defined because the projection functor is strong monoidal.  Furthermore,
  this monoidal structure is \culf thanks to the above compatibility: to give a
  pair of $R$-structures with a layering of each is the same as giving a pair of
  $R$-structures with a layering of its disjoint union. 
    This is to say that this square is a pullback:
  $$\xymatrix{
     \R^\iso_{/\un 1} \times \R^\iso_{/\un 1} \ar[d] & \ar[l]_{g\times g} 
	 \R^\iso_{/\un k}\times \R^\iso_{/\un k} \ar[d] \\
     \R^\iso_{/\un 1} & \ar[l]^g \R^\iso_{/\un k} ,
  }$$
  where $g$ is the unique generic map (and $k$ could be $0$).
\end{proof}
It follows that every monoidal restriction species defines a 
bialgebra (a Hopf algebra in the connected case), and a morphism of monoidal 
restriction species defines a bialgebra homomorphism.

\begin{blanko}{Remark.}
  There is a kind of converse to the construction $R \leadsto \ds R$.
  Namely, starting from a
  decomposition space $\ds R$ \culf over $\ds I$ (and with $\ds R_1$ locally
  finite), we can take lower dec of
  both and obtain a Segal space which by Lemma~\ref{lem:DecULF-is-right} is
  a right fibration over $\Dec_\bot \ds I = \fatnerve\I$ 
  (Proposition~\ref{prop:dec-bb-is-ii}).  In fact $\Dec_\bot \ds R$ is a
  Rezk-complete Segal space.  Indeed, since $\ds R$ is \culf over $\ds I$,
  it is complete, locally finite, locally discrete and of locally
  finite length, by Lemma~\ref{lem:culf-local}.  But also the dec map
  $\Dec_\bot \ds R \to \ds R$ is \culf, so $\Dec_\bot \ds R$ also has all
  these properties.  Since it is furthermore a Segal space, it follows from
  a general result of \cite{GKT:MI} that it is Rezk complete.
  Hence $\Dec_\bot \ds R$ is essentially the fat nerve of a category $\R$
  (with a right fibration over $\I$).
  
%STRICT%   One can circumvent the subtleties by subjecting
%STRICT%   the input data $\ds R \to \ds I$ to suitable strictness conditions:
%STRICT%   require $\ds R$ to be \culf over $\ds I$ by {\em strict} pullbacks,
%STRICT%   whereby $\ds R$ itself becomes a strict decomposition space in the sense
%STRICT%   that the generic-free pullback squares are strict pullbacks. It follows
%STRICT%   that $\Dec_\bot \ds R$ is a strict category object, and the right fibration 
%STRICT%   $\Dec_\bot \ds R \to \fatnerve \I$ is split.  Altogether it is itself
%STRICT%   the fat nerve a split cartesian fibration $\R\to \I$, on the nose.
\end{blanko}

%%%%%%%%%%%%%%%%%%%%%%%%%%%%%%%%%%%%%%%%%%%%%%%%%%
\section{The decomposition space $\ds C$ of layered finite posets}
%%%%%%%%%%%%%%%%%%%%%%%%%%%%%%%%%%%%%%%%%%%%%%%%%%

\label{sec:C}

We define and study the monoidal decomposition space $\ds C$ of finite posets
and their `admissible cuts', which will play the same role for
directed restriction species as $\ds I$ does for plain restriction species.
An important difference is that while the simplicial groupoid $\ds I$ is a Segal
space, $\ds C$ is only a decomposition space, not a Segal space.

\begin{blanko}{Convex maps of posets.}
  A subposet $K$ of a poset $P$ is {\em convex} if it is full and if $a\leq
  x\leq b$ in $P$ and $a,b\in K$ imply $x \in K$.  A map of posets $f:K \to
  P$ is {\em convex} if for all $a,b\in K$ and $fa\leq x \leq fb$ in $P$
  there is a unique $k\in K$ with $a\leq k \leq b$ and $fk=x$.  In other
  words, $f$ is injective and $f(K) \subset P$ is a convex subposet.  We
  denote by $\C$ the category of finite posets and convex maps.
\end{blanko}

\begin{lemma}\label{pbkconv}
  In the category of posets, convex maps are stable under pullback.
\end{lemma}

\begin{lemma}\label{lem:convex-characterisation}
  For a subposet $K \subset P$ the following are equivalent.
  \begin{enumerate}
    \item $K$ is convex
    
    \item $K$ is the middle fibre of some monotone map $P \to \un 3$
	
	\item $K \subset P$ is a fully faithful
	ULF functor of categories.
  
  \end{enumerate}  
\end{lemma}

\begin{blanko}{Layered posets.}\label{polayering}
  An {\em $n$-layering} of a finite poset $P$ is a monotone map $\ell: P\to\un n$.
  We refer to the fibres $P_i = \ell^{-1}(i)$, $i\in\un n$, as {\em layers}.
  Layers are convex subposets, by the previous lemma, and may be empty.
  
  For sets, considered as discrete posets, the notion of set layering from
  \ref{layering} agrees with the notion of poset layering.  Poset layering
  is more subtle, however, as it contains more information than just the
  list of layers.
\end{blanko}

\begin{blanko}{The groupoid of $n$-layered finite posets.}\label{C/k}\label{C/n}
  Consider the groupoid $\C^\iso_{/\un n}$ of $n$-layerings of finite posets.
  That is, the objects of $\C^\iso_{/\un n}$ are monotone maps $\ell:P \to\un n$,
  and the morphisms are commutative triangles
  $$
  \xymatrix@C=2ex{P\drto\rrto^\simeq&&P'\dlto\\&\un n,}
  $$
  where $P \isopil P'$ is a monotone bijection (a poset isomorphism).
\end{blanko}

\begin{blanko}{The simplicial groupoid of layered finite posets.}
  We can define face and degeneracy maps between the groupoids of layered finite
  posets to assemble them into a simplicial groupoid $\ds C$,
  in the same way as for layered finite sets in \ref{simp-layering}:
  
  The degeneracy and the inner face maps are defined using the 
  correspondence $\Deltagen\op \simeq \unDelta$: if $g:[n]\to[m]$ 
  is a generic map in $\simplexcategory$ then 
  $g^*:\C^\iso_{/\un m}\to\C^\iso_{/\un n}$ 
  is given by postcomposition with the
  corresponding map $\un g:\un m\to\un n$ in $\un\simplexcategory$,
  $$
  P {\to} \un m \qquad \mapsto  \qquad P {\to} \un m {\to} \un n.
  $$
  The definition for free maps (composites of outer face maps) is by pullback:
  for example, $d_\top : \C^\iso_{/\un n} \to \C^\iso_{/\un{n{-}1}}$ is given
  by taking $P' {\to} \un n$ to  $P {\to} \un{n{-}1}$ in the
  pullback square
  $$\xymatrix{
	 P \drpullback\ar[r]\ar[d] & P' \ar[d] \\
	 \un {n{-}1} \ar[r]_-{\un d_\top} & \un n .
  }$$
  Since $\un d_\top : \un{n{-}1} \to \un n$ is a convex map of
  posets, so is $P\to P'$. To be explicit, we can take this
  convex map to be an actual subset inclusion.
\end{blanko}

\begin{prop}
  The groupoids $\C^\iso_{/\un n}$ and the maps between them, defined
  above, form a 
%STRICT%   strict 
  simplicial groupoid $\ds C$.
\end{prop}

\begin{proof}
  The check may be performed in precisely the same way as done for
  $\ds I$ in Proposition~\ref{B/k}: one checks the constructions above
  are covariantly functorial in all maps in $\unDelta$ (giving the generic
  part), contravariantly functorial in the convex maps of $\unDelta$
  (giving the free part), and that these two functorialities are compatible.
  We will formalise this later in the notions of $\newnabla$-spaces
  and sesquicartesian fibrations
  (Sections~\ref{sec:nabla}--\ref{sec:sesq}).
\end{proof}

\begin{blanko}{Lower-set inclusions.}
  Let $P$ be a poset.  A full subposet $L\subset P$ is a {\em lower set}
  (also called an {\em ideal}) if $x\leq b$ in $P$ and $b\in L$ imply $x
  \in L$.  A map of posets $L\to P$ is a {\em lower-set inclusion} if it is
  injective, full, and its image is a lower set in $P$.  Clearly lower-set
  inclusions are convex.  Let $\C^\low$ denote the category of finite
  posets and lower-set inclusions.  Note that $L \to P$ is a lower-set
  inclusion if and only if it is a right fibration of categories.  Upper
  sets are defined analogously.
\end{blanko}

\begin{lemma}\label{lem:lsi-pbk}
  In the category of posets, lower-set inclusions are stable under pullback.
\end{lemma}

\begin{prop}
  The map $\un d_\top:\un 1\to\un 2$ classifies lower-set inclusions. 
  That is, if $P$ is a poset, pullback along $\un d_\top$ defines a bijection
  $$
  \{\mbox{monotone maps }P\to\un2\} \;\;\cong \;\;\{\mbox{isoclasses of lower-set inclusions }L\subseteq P\}.
  $$
\end{prop}

\begin{prop}\label{prop:Dec(C)}
  There are natural (levelwise) equivalences
  $$
  \Dec_\bot(\ds C) \simeq \fatnerve \C^\low
  \qquad
  \Dec_\top \ds C \simeq \fatnerve (\C^\upper)\op
  $$
\end{prop} 
\begin{proof}
  There is a natural equivalence  
  $$
  \C^\iso_{/\un n} 
 \;\; \simeq\;\; 
  \Map([n{-}1],\C^\low)
  $$
  Given an $n$-layering of a poset $P$ (i.e.~a monotone map $P\to\un n$),
  let $P_{\un n}=P$ and define inductively $P_{\un k}\to \un k$ as the
  pullback of $P_{\un {k{+}1}}\to \un {k{+}1}$ along the lower-set inclusions
  $\un d_\top:\un{ k} \to \un {k{+}1}$.  By Lemma \ref{lem:lsi-pbk}, we
  obtain lower-set inclusions $P_{\un k}\to P_{\un {k{+}1}}$.  Then the
  equivalence assigns to $P {\to} \un n$ the sequence of lower-set inclusions
  $$
  (P_{\un1} \into P_{\un2} \into \cdots \into P_{\un{n{-}1}} \into P)
  \;\;\in\;\; \Map([n{-}1],\C^\low).
  $$
  This assignment is fully faithful since each automorphism of such
  sequences corresponds to a unique automorphism of $P$ over $\un n$.
  Finally, given such a sequence of lower-set inclusions, we recover a
  monotone map $P\to \un n$, sending $x$ to the least $k$ for which $x\in
  P_{\un k}$.
  It is straightforward to check that the face maps match up as required, so as
  to assemble these equivalences into a levelwise equivalence of simplicial 
  groupoids.

  The result for the upper dec is analogous.  The `op' appears in that case
  because the smallest subset in the chain is the last one, not the first
  as above.
\end{proof}

\begin{prop}
  $\ds C$ is a decomposition space (but not a Segal space).
\end{prop}
\begin{proof}
  We apply the decalage criterion~\cite[Theorem~4.11~(4)]{GKT:DSIAMI-1}. 
  We already proved 
  that the two Decs are Segal spaces.  It remains to check that the following
  two squares are pullbacks:
  $$\xymatrix{
     \C^\iso_{/\un 0} \ar[r]^{s_0} & \C^\iso_{/\un 1} \\
     \C^\iso_{/\un 1} \ar[u]^{d_\bot} \urpullback \ar[r]_{s_1} & \C^\iso_{/\un 2} \ar[u]_{d_\bot}
     }
     \qquad
     \xymatrix{
     \C^\iso_{/\un 0} \ar[r]^{s_0} & \C^\iso_{/\un 1} \\
     \C^\iso_{/\un 1} \ar[u]^{d_\top} \urpullback \ar[r]_{s_0} & \C^\iso_{/\un 2} 
     \ar[u]_{d_\top}
  }
  $$
  But it is clear they are strict pullbacks: this amounts to
  saying that if a $2$-layered poset has one layer empty, it is
  determined by the other layer.  Since the free face maps are
  iso-fibrations, the squares are also (homotopy) pullbacks.  Clearly $\ds C$
  is not a Segal space as $\C^\iso_{/\un{m{+}n}} \not\simeq
  \C^\iso_{/\un{m}} \times \C^\iso_{/\un{n}}$.
\end{proof}

\begin{lemma}\label{lem:Cfinite}
  The decomposition space $\ds C$ is complete, locally finite and locally discrete,
  and of locally finite length.
\end{lemma}

\begin{proof}
  Since $\C^\iso_{/\un0}$ is contractible, consisting of the empty poset 
  with no non-trivial automorphisms, we know $s_0:\C^\iso_{/\un0}\to\C^\iso_{/\un1}$ 
  is mono, so $\ds C$ is complete. Now observe that $\C^\iso_{/\un1}$ is locally 
  finite as each finite poset has only finitely many automorphisms. 
  We have just seen that $s_0:\C^\iso_{/\un0}\to\C^\iso_{/\un1}$ is finite and 
  discrete, and for $d_1:\C^\iso_{/\un2}\to\C^\iso_{/\un1}$ the fibre over each 
  finite poset $P$ is the finite discrete groupoid $\{P\to\un2\}$ of all 
  monotone maps. 
  Lastly, $\ds C$ is of locally finite length: the degenerate simplices are 
  precisely the layerings with an empty layer.  The fibre of
  $g:\C^\iso_{/\un n}\to\C^\iso_{/\un 1}$ over $P$ has no non-degenerate 
  simplices
  if $n$ is greater than the number of elements of the finite poset $P$.
\end{proof}

\begin{lemma}\label{lem:CC-monoidal}
  $\ds C$ is a monoidal decomposition space under disjoint union.
\end{lemma}

\begin{proof}
  For fixed $k$, we have $\C^\iso_{/\un k} \times \C^\iso_{/\un k} \to \C^\iso_{/\un k}$ given by 
  disjoint union.  It is clear that these maps assemble into a simplicial map
  $\ds C \times \ds C \to \ds C$. CULFness of this simplicial map
  follows because to give a pair 
  of posets, each with a $k$-layering,
  is the same as giving a pair of
  posets, together with a $k$-layering of their disjoint union.
  In other words, disjoint union of layered posets are computed layer-wise.
  Diagrammatically, this square is a pullback:
  $$\xymatrix{
     \C^\iso_{/\un1} \times \C^\iso_{/\un1} \ar[d] & \ar[l]_{g\times g} \C^\iso_{/\un k}\times \C^\iso_{/\un k} \ar[d] \\
     \C^\iso_{/\un1} & \ar[l]^g \C^\iso_{/\un k} ,
  }$$
  where $g$ is the unique generic map (and $k$ could be $0$).
\end{proof}

%%%%%%%%%%%%%%%%%%%%%%%%%%%%%%%%%%%%%%%%%%%%%%%%%%
\section{Directed restriction species}
%%%%%%%%%%%%%%%%%%%%%%%%%%%%%%%%%%%%%%%%%%%%%%%%%%
\label{sec:DRSp}

We introduce the new notion of 
directed restriction species, with associated
incidence coalgebras 
generalising well-known constructions with rooted
forests~\cite{Dur:1986,Connes-Kreimer:9808042}, acyclic directed
graphs~\cite{Manchon:MR2921530,Manin:0904.4921}, posets and
distributive lattices~\cite{Schmitt:1994,Figueroa-GraciaBondia:0408145}, and double 
posets~\cite{Malvenuto-Reutenauer:0905.3508}.

\begin{blanko}{Directed restriction species.}
  A {\em directed restriction species} is by definition a 
  (pseudo)-functor
  $$
  R:\C\op\to\Grpd ,
  $$
  or equivalently, by the Grothendieck construction, a 
%STRICT%   split 
  right fibration
  $\R\to\C$.  We shall always assume that all values are locally 
  finite groupoids. 
  
  The idea is that the value on a poset $S$ is the groupoid of
  all possible $R$-structures that have $S$ as underlying poset.
  
  A morphism of directed restriction species is just a 
  (pseudo)-natural transformation.
  This defines the category of directed restriction 
  species $\kat{DRSp}$, equivalent to the categories of 
  groupoid-valued presheaves on $\C$, and of right fibrations over $\C$:
  $$
  \kat{DRSp}\simeq\Grpd^{\C\op}\simeq\kat{RFib}_{/\C} .
  $$ 
\end{blanko}

\begin{blanko}{Coalgebras from directed restriction species.}
  Let $R$ be any directed restriction species.  An {\em admissible cut} of
  an object $X\in R[P]$ is by definition a $2$-layering of the underlying
  poset.  In other words, the cut separates $P$ into a lower-set and an
  upper-set.  This agrees with the notion of admissible cut in
  Butcher--Connes--Kreimer (as in \ref{ex:CK} above), and in related
  examples.
    
  A coalgebra is defined by the rule
  \begin{equation}\label{eq:comultRS}
  \Delta(X) = \sum_{c\in \operatorname{cut}(P)}  
  X|D_c \tensor X|U_c,  \qquad X \in R[P],
  \end{equation}
  where the sum is over all admissible cuts $c=(D_c,U_c)$.
  
  Note that the incidence coalgebra of a directed restriction species is generally
  non-cocommutative.  It is cocommutative if and only if it is 
  actually supported on discrete posets, so that in reality it is an 
  ordinary restriction species, as we explain next.
\end{blanko}

\begin{blanko}{Sets as discrete posets.}\label{RSpsubDRSp}
  Any finite set can be regarded as a discrete poset, and any
  injective map of sets is then a convex map.  Hence there is a
  natural functor $\I \to \C$.  This functor is easily seen to be a
  right fibration.  Hence every restriction species is also a directed
  restriction species.  This is to say that there is a natural functor
  $$
  \kat{RSp} \to \kat{DRSp}
  $$
  from restriction species to directed
  restriction species, clearly fully faithful. 
\end{blanko}

\begin{blanko}{Directed restriction species as decomposition spaces.}
  If $\R\to\C$ is a directed restriction species, let $\ds R_k$ be the
  groupoid of $R$-structures on posets $P$ with a $k$-layering.
  (In other words, $\ds R_2$ is the groupoid of
  $R$-structures with an admissible cut, and $\ds R_k$ is the groupoid of
  $R$-structures with $k-1$ compatible admissible cuts.)
\end{blanko}

\begin{theorem}\label{DRisDS}
  The $\ds R_k$ form a simplicial groupoid $\ds R$, which is a
  decomposition space.  Morphisms of directed restriction species induce
  \culf functors between decomposition spaces. The construction 
  defines a functor from the category of directed restriction species and 
  their morphisms to that of decomposition spaces and \culf maps.
\end{theorem}

\begin{proof}
  This can be proved in the same way as Theorem \ref{RisDS} for ordinary restriction
  species, or a more elegant proof will be given in Theorem~\ref{thm:RSP&DRSp->DS}, after setting up fancier machinery.
\end{proof}

Since we assume  directed restriction species $R: \C\op\to\Grpd$
take
locally finite groupoids as values, it follows by Lemma~\ref{lem:R1} that 
$\ds R_1$ is a locally finite groupoid.  Now by Lemmas~\ref{lem:culf-local} 
and~\ref{lem:Cfinite} we have the necessary finiteness
conditions to obtain classical incidence coalgebras by taking homotopy 
cardinality:

\begin{lemma}\label{lem:DRplocdiscrete}
  The decomposition space $\ds R$ is complete, locally finite, locally discrete,
  and of locally finite length.
\end{lemma}

\begin{lemma}
  The incidence coalgebra obtained by taking homotopy cardinality 
  coincides with formula~\eqref{eq:comultRS}.
\end{lemma}

\begin{proof}
  The main point here is that since $\ds R$ is locally discrete by 
  Lemma~\ref{lem:DRplocdiscrete}, the 
  homotopy sum resulting from the decomposition space is just an ordinary sum, 
  as in \eqref{eq:comultRS}.
\end{proof}

\begin{blanko}{Monoidal directed restriction species.}\label{monoidaldirected}
  The category $\C$ is symmetric monoidal under disjoint union.  We define
  a {\em monoidal directed restriction species} to be a directed
  restriction species $\R\to\C$ for which the total space $\R$ has a
  monoidal structure and the right fibration is also a strong monoidal
  functor.
  This extends the notion of ordinary monoidal restriction species
  introduced in \ref{monoidalrestr}, as $\I \to \C$ is easily seen to be a
  monoidal directed restriction species.  Since strong monoidal right
  fibrations compose, every monoidal restriction species is also a monoidal
  directed restriction species.
  We have:
\end{blanko}

\begin{prop}\label{prop:mdrsp-mds}
  The functor of Theorem~\ref{DRisDS} extends to a functor from monoidal
  directed restriction species and their morphisms, to monoidal
  decomposition spaces and \culf monoidal functors.
\end{prop}

If a restriction species is monoidal, the associated incidence
coalgebra becomes a bialgebra.  The projection $\ds R \to \ds C$ is
monoidal, and so the incidence bialgebra of $\R$ comes with a bialgebra
homomorphism to the incidence bialgebra of $\C$.

\bigskip

Except when explicitly mentioned otherwise, all the following examples are in fact monoidal directed restriction 
species and hence induce bialgebras.

\begin{blanko}{First examples.}\label{ex:DRSp}
  Just as for ordinary restriction species, it is sometimes useful to describe 
  a directed restriction species by describing the associated 
  right fibration $\R \to\C$, where the restriction structure is 
  encoded in the arrows.
  
  (1) {\em Posets.} The category $\C$ of finite posets and convex maps is
  the terminal directed restriction species.  The resulting coalgebra
  comultiplies a poset by splitting it along `admissible cuts' into
  lower-sets and upper-sets (cf.~Example~\ref{ex:posetcoalg}).
  
  (2) {\em One-way categories and M\"obius categories.} For a finite
  category $\CC$ to have an underlying poset, it is required that for any
  two objects $x,y \in \CC$ at least one of the hom sets $\Hom_{\CC}(x,y)$
  and $\Hom_{\CC}(y,x)$ is empty.  (This implies that $\CC$ is skeletal.)
  The underlying poset $\underline\CC$ is then given by declaring $x\leq y$
  to mean that $\Hom_{\CC}(x,y)$ is nonempty.  Such categories form a
  directed restriction species $U$: for a convex map of posets $K
  \subset \underline\CC$, the restriction of $\CC$ to $K$ is given as the
  full subcategory spanned by the objects in $K$.  For the corresponding
  right fibration $\mathbb{U}\to\C$, the arrows in $\mathbb{U}$ are the
  fully faithful \culf functors (automatically injective on objects since
  the categories are skeletal).
  
  With the further condition imposed that the only endomorphisms are the
  identities, we arrive at the notion of \emph{finite delta}, in the
  terminology of Mitchell~\cite{Mitchell:RSO}, now more commonly called
  finite \emph{one-way categories}.  This is equivalent
  (cf.~\cite{LawvereMenniMR2720184}) to the notion of finite \emph{M\"obius
  category} of Leroux~\cite{Leroux:1975}.  M\"obius categories play an
  important role as a generalisation of locally finite posets, and in
  particular admit M\"obius inversion.  It is clear that we also have a
  directed restriction subspecies of finite M\"obius categories.
\end{blanko}

\begin{blanko}{Convex-closed classes of posets.}
  Ordinary (restriction) species are mostly about structure, not property,
  since the only property that can be assigned to a finite set is its
  cardinality.  For directed restriction species, property plays a more
  important role, since posets can have many properties.  Any class of
  posets closed under taking convex subposets and closed under isomorphisms
  defines a (fully faithful) right fibration, and hence a directed
  restriction species.  Such a class may or may not be monoidal
  under disjoint union. (Note that this notion, which could reasonably be
  called convex-closed classes of posets, is different from the classical
  closure property in incidence coalgebras, where a class of {\em intervals}
  is required to be closed under subintervals~\cite{Schmitt:1994}.)

  For example, forests (cf.~\ref{ex:trees} below), linear orders, and 
  discrete posets (cf.~\ref{RSpsubDRSp}) are convex-closed classes of posets,
  and form (monoidal) directed restriction species.
  Considering linear orders leads to $\mathbb{L}$-species, in the sense of
  \cite{Bergeron-Labelle-Leroux}.

  Just as in the case of ordinary restriction species, the minimal such
  `ideals' are defined by picking any single poset $P$, and considering the
  `principal ideal generated by $P$', more precisely the slice category
  $\C_{/P}$.  Note that $\C_{/P}$ cannot be monoidal in the sense 
  of \ref{monoidaldirected}.
  Since the morphisms in $\C$ are just the convex
  maps, $\C_{/P}$ is equivalent to the full subcategory
  of $\C$ consisting of $P$ and all its convex subposets.  This reflects
  the standard fact that any element in a coalgebra spans a subcoalgebra.
\end{blanko}

\begin{blanko}{Examples: various flavours of trees (actually forests).}\label{ex:trees}
  (1) {\em Combinatorial trees.} 
  Consider the directed restriction species of rooted forests:
  a rooted forest has an underlying poset, whose convex
  subposets inherit each a rooted-forest structure.
  Regarded as a right fibration $\mathbb H \to \C$, the category 
  $\mathbb H$ has objects rooted forests and morphisms subforest
  inclusions (not required to preserve the root).
  The resulting bialgebra is the Butcher--Connes--Kreimer
  Hopf algebra~\cite{Dur:1986,Connes-Kreimer:9808042} already
  treated in \ref{ex:CK}.  As explained, this is not a Segal
  groupoid: a tree cannot be reconstructed from its layers.
  An important non-commutative variation comes from planar 
  forests~\cite{Foissy:2002I}.

  (2) {\em Operadic trees (with nodes).} Consider the combinatorial
  structure of rooted forests allowing open-ended edges (leaves and root)
  as in \cite{Kock:0807,GalvezCarrillo-Kock-Tonks:1207.6404}, but
  disallowing isolated edges, i.e.~edges not adjacent to any node.  As
  before, each such forest has an underlying poset of nodes, and for each
  convex subset of the node set, there is induced a forest again.  These
  are full forest inclusions, meaning that for each node, all incoming
  edges as well as the outgoing edge must be included (see \cite{Kock:0807}
  for details).  It is an important feature that the local structure at the
  nodes is always preserved under taking such subforests.  This means that
  one can consider trees whose nodes are decorated with `operation symbols'
  of matching arity (more precisely $P$-trees for $P$ a polynomial
  endofunctor~\cite{Kock:0807,Kock:MFPS28}) 
  and that subtrees inherit
  such decorations.  This is not possible for combinatorial trees, where
  the cuts destroy the local structure of nodes (such as for example being
  a binary node).  Operadic forests (with nodes) form a directed
  restriction species.  Note that in contrast to what happens for
  combinatorial trees, cuts do not delete inner edges, they cut them in two (as
  a consequence of the fullness of subforest inclusions).  But if an 
  isolated edge results from a cut, it is deleted, as illustrate in this
  figure:
  \begin{center}\begin{texdraw}
	  \move (-5 -5)
	  \bsegment
	  \move (0 0) \lvec (0 10) \onedot \lvec (-5 22) \onedot \lvec 
	  (-9 30) \move (-5 22) \lvec (-4 30) \move (0 10) \lvec (1 30)
    \move (0 10) \lvec (7 30)
	\move (-10 15) \clvec (-4 17)(4 17)(10 15)
	\esegment
	
	\htext (30 10){$\leadsto$}
	
	\move (50 0) 
	\bsegment
	\move (0 10) \lvec (0 20) \onedot \lvec (-5 30) \move (0 20) 
	\lvec (5 30)
	\esegment
	
	\move (70 -10)
	\bsegment
	\move (0 0) \lvec (0 10) \onedot \lvec (-7 20) \move (0 10) 
	\lvec (0 20) \move (0 10) \lvec (7 20)
	\esegment
	\end{texdraw}
	\end{center}

  (3) {\em Non-example: operadic trees, including nodeless ones.} If one
  allows the nodeless tree, the resulting notion of forest does not form a
  directed restriction species.  Indeed, with all the nodeless forests
  being different structures on the empty set of nodes, and since there
  exist non-invertible maps between such node-less forests, the functor to
  $\C$ cannot be a right fibration (it has non-invertible arrows in its
  fibres).  (It is only over the empty set that this problem arises: for
  trees with nodes, every non-invertible map can be detected on nodes.)

  This variation, which is subsumed in the class of decomposition
  spaces coming from operads~\cite{GKT:ex,Kock-Weber:1609.03276}, has
  some different features which have been exploited to good effect in
  various contexts
  \cite{GalvezCarrillo-Kock-Tonks:1207.6404,Kock:1109.5785,Kock:1411.3098,Kock:1512.03027}.
  In particular it is important that the cut locus expresses a type
  match between the roots of the crown forest and the leaves of the
  bottom tree, and that there is a grading \cite{GKT:DSIAMI-2} given
  by number of leaves minus number of roots.  The incidence bialgebra
  is not connected: the zeroth graded piece is spanned by the
  node-less forests.  These are all group-like, and the connected
  quotient (dividing out by this coideal) is precisely the incidence
  Hopf algebra of the directed restriction species of forests without
  isolated edges.  One can then further take \emph{core}
  \cite{Kock:1109.5785,Kock:1512.03027}, which means
  shave off leaves and root (and forget the $P$-decoration).  This is
  a monoidal \culf functor, and altogether there is a monoidal \culf
  functor from the decomposition space of $P$-trees to the
  decomposition space of combinatorial trees.
  This is an
  interesting example of a relative $2$-Segal space in the sense of
  Young~\cite{Young:1611.09234} and Walde~\cite{Walde:1611.08241}.
\end{blanko}

\begin{blanko}{Examples: various flavours of acyclic directed graphs.}
  (1) {\em Acyclic directed graphs.} These have underlying posets,
  where $x\leq y$ if there is a directed path from $x$ to $y$.  Any
  convex subposet of the poset of vertices induces a subgraph $S$,
  which is convex in the usual sense of directed graphs, meaning that
  any directed path from $x\in S$ to $y\in S$ in the whole graph must
  be entirely contained in $S$.  There is now induced a natural notion
  of admissible cut, similar to Butcher--Connes--Kreimer, and a Hopf
  algebra results (see Manchon~\cite[\S5]{Manchon:MR2921530}).

  (2) {\em Acyclic directed open graphs.} Now we allow open-ended edges,
  thought of as input edges and output edges (see~\cite{Kock:1407.3744}), but
  we do not allow graphs containing isolated edges.  This situation and the
  resulting bialgebra have been studied by 
  Manchon~\cite[\S4]{Manchon:MR2921530}.  Interesting decorated versions have been studied by
  Manin~\cite{Manin:MR2562767,Manin:0904.4921} in the theory of
  computation.  His graphs are decorated by operations on partial recursive
  functions and switches.

  (3) {\em Non-example: Acyclic directed open graphs, allowing isolated 
  edges.} Again, if one allows isolated edges, it
  is not a restriction species.  In contrast it is a Segal groupoid, and
  the comultiplication resulting from it enjoys a nice grading (by number
  of input edges minus number of output edges).
\end{blanko}

\begin{blanko}{Examples: double posets and related structures.}
  A {\em double poset}~\cite{Malvenuto-Reutenauer:0905.3508} is a
  poset $(P,\leq)$ with an additional poset structure $\preccurlyeq$, not
  required to have any compatibility with $\leq$.  Let $\D$ denote the
  category of finite double posets $(P,\leq,\preccurlyeq)$ and inclusions
  that are convex for $\leq$.  For every $\leq$-convex subset $(K,\leq)
  \subset (P,\leq)$, there is induced a $\preccurlyeq$ structure on $K$,
  simply by the fact that posets form an ordinary restriction species
  (cf.~\ref{ex:posets}~(3)).  It follows that $\D \to \C$ is a right
  fibration, and hence a directed restriction species.  The associated 
  incidence coalgebra was first studied by Malvenuto and 
  Reutenauer~\cite{Malvenuto-Reutenauer:0905.3508}; see 
  \cite{Foissy:MR3016301} and \cite{Foissy:MR3046302}
  for more recent developments.
  
  The case where the second poset structure is a linear order is called
  {\em special double poset} or just {\em special poset}, and is 
  equivalent to Stanley's notion of labelled poset~\cite{Stanley:Mem1972}.
  
  Double posets and special posets are just two instances of the following
  general construction: for any ordinary restriction species $R$, consider
  the directed restriction species consisting of having simultaneously a
  poset structure and an $R$-structure, without compatibility conditions.
  Let the morphisms be inclusions that are convex for the poset structure.
\end{blanko}

\begin{blanko}{Decalage.}\label{DecR=NRlow}
  While for ordinary restriction species $\R\to\I$
  we have $\Dec_\bot \ds R  \simeq \fatnerve \R$ and 
  $\Dec_\top \ds R \simeq \fatnerve \R\op$, the situation is slightly 
  more complicated for directed restriction species.  The result is 
  (as we shall see in Proposition~\ref{prop:Dec-of-DRSp}):
  $$
  \Dec_\bot \ds R \simeq \fatnerve \R^\low 
  \qquad
  \qquad
  \Dec_\top \ds R \simeq \fatnerve (\R^\upper)\op 
  $$
  where $\R^\low \subset \R$ denotes the subcategory of 
  $R$-structures with all the objects, but only the maps whose 
  underlying poset map is a lower-set inclusion.  (Similarly, 
  $\R^\upper$ has only upper-set inclusion.)
  (Note that this result does not contradict \ref{DecR=NR}: if an ordinary
  restriction species $\R$ is considered a directed restriction species (as
  in \ref{RSpsubDRSp}) supported on discrete posets, then all inclusion
  maps are both lower-set inclusions and upper-set inclusions.)
  
  This result is interesting because it relates to classical
  reduced-incidence-coalgebra constructions.  Recall from
  Example~\ref{ex:CK} that D\"ur~\cite{Dur:1986} constructs the
  Butcher--Connes--Kreimer Hopf algebra as the reduced incidence
  coalgebra of the (opposite of the) category of rooted forests and
  root-preserving inclusions.  The reduction identifies two
  forests inclusions if they have isomorphic complement crowns.  The
  reduction is now seen to be the upper-dec map, since the underlying 
  poset of a forest is oriented from leaves to roots, so the
  root-preserving inclusions are the upper-set inclusions.
\end{blanko}

%%%%%%%%%%%%%%%%%%%%%%%%%%%%%%%%%%%%%%%%%%%%%%%%%%
\section{Convex correspondences and `nabla spaces'}
%%%%%%%%%%%%%%%%%%%%%%%%%%%%%%%%%%%%%%%%%%%%%%%%%%
\label{sec:nabla}

\begin{blanko}{Convex correspondences.}
  Consider the category $\newnabla$ of {\em convex correspondences} in 
  $\unDelta$, a subcategory of the category of spans in
  $\unDelta$.
  Objects are those of $\unDelta$, and morphisms are spans
  $$
  \xymatrix{
  \un n' & \ar@{ >->}[l]_j \un n \ar[r]^f & \un k}
  $$
  where $j$ is convex.
  Composition of such spans is given by pullback, which exist by Lemma~\ref{convexpbk}.
  By construction, $\newnabla$ has a factorisation system in which 
  the left-hand
  class (called {\em backward convex} maps) consists of spans of the form $\xymatrix{\cdot & \ar@{ >->}[l] \cdot 
  \ar[r]^= &\cdot}$, and the right-hand class (called {\em ordinalic} maps)
  consists of spans of the form $\xymatrix{\cdot & \ar[l]_= \cdot 
  \ar[r] &\cdot}$.
  Composition of an ordinalic map followed by backward convex map is defined by
  \begin{equation}\label{nabla-comp}
  (\xymatrix{\cdot & \ar@{ >->}[l]_i \cdot \ar[r]^= &\cdot})
  \circ
  (\xymatrix{\cdot & \ar[l]_= \cdot \ar[r]^g &\cdot})
  \;\;=\;\;
  (\xymatrix{\cdot & \ar@{ >->}[l]_j \cdot \ar[r]^f &\cdot})
  \end{equation}
  with reference to the pullback square
  $$\xymatrix{
     \cdot \ar[d]_g & \cdot \dlpullback \ar[d]^f \ar@{ >->}[l]_{j} \\
     \cdot  & \ar@{ >->}[l]^i \cdot
  }$$
\end{blanko}

\begin{lemma}\label{DeltaNabla}
  There is a canonical functor
  $$\gamma:\simplexcategory\op
  \longrightarrow \newnabla,\qquad[n]
  \longmapsto\un n,$$
  restricting to isomorphisms 
\begin{equation}\label{isos-hahaha}
  \Deltagen\op  \cong 
  \unDelta \cong  \newnabla_{\mathrm{ordinalic}},
\quad \quad
(\simplexcategory_{\mathrm{free}}^{\geq1})\op  \cong 
(\unDelta_{\mathrm{convex}}^{\geq1})\op \cong 
\newnabla_{\mathrm{back.conv.}}^{\geq1} ,
\end{equation}  
  and sending all maps $[0]\to[n]$ in $\simplexcategory$
  to the zero map $\un n \lat \un 0 \to \un 0$ in $\newnabla$.
In particular, $\gamma$ is bijective on objects and full.
\end{lemma}
\noindent
In summary, the categories $\simplexcategory\op$
and $\newnabla$ differ only in the fact that
$\un 0\in\newnabla$ is initial {\em and terminal}, whereas
$\Hom_{\simplexcategory\op}([n],[0])$ 
contains $n+1$ maps.

\begin{proof}
  The first isomorphism is Lemma~\ref{Joyal-duality} and the second 
  is Lemma~\ref{lem:canonical-iso-hahaha}.
\end{proof}

\begin{prop}
  Precomposing with the canonical functor $\gamma: \simplexcategory\op 
  \to \newnabla$ of Lemma~\ref{DeltaNabla}
  induces a fully faithful functor
  $$
  \gamma\upperstar:\Fun(\newnabla,\Grpd) \to\Fun(\simplexcategory\op,\Grpd)
  $$ 
  whose essential image is the full subcategory
  consisting of simplicial objects with $d_\bot=d_\top:X_1\to X_0$.
\end{prop}

\begin{proof}
  Any functor which is bijective on objects and full induces a
  fully faithful functor of the presheaf categories.  The main point 
  is to characterise the essential image.  
  Note that every simplicial object $X$ in the image will have all maps
  $X_n\to X_0$ equal, since the functor $\gamma$ sends all maps $[0]\to[n]$
  to the same image.  Given a simplicial object $X$ with all $X_n\to X_0$
  equal, we define a $\newnabla$-diagram by sending each object $\un n$ to
  $X_n$ and sending each convex correspondence $\un n' \stackrel{j}\lat 
  \un n \stackrel{f}\to \un k$
  to the composite
  $$
  \xymatrix{ X_{n'} \ar[rr]^{X(\gamma^{-1}(j))} && X_n 
  \ar[rr]^{X(\gamma^{-1}(f))} && X_k, }
  $$
  assuming $n>0$ so as to invoke the bijections \eqref{isos-hahaha} separately
  on backward convex and ordinalic maps.  For $n=0$, $\gamma^{-1}(j)$
  is not well defined in $\simplexcategory$, but taking $X$ on it {\em is}
  well defined, since we have assumed all the maps $X_n \to X_0$ coincide.
  To check functoriality of the assignment, it is enough to treat the
  situation of an ordinalic map followed by a backward convex map.  These
  compose by pullback in $\unDelta$, and by Lemma~\ref{lem:free-gen-Delta}
  these pullback squares correspond to commutative squares in
  $\simplexcategory$, in a way compatible with the assignments on arrows,
  so as to ensure that composition is respected.
  It is clear that this nabla space induces $X$ as required.
\end{proof}

\begin{blanko}{Iesq condition on functors.}\label{iesq-on-functors}
  For a functor $X:\newnabla \to \Grpd$, the image of a backward convex map
  is denoted by upperstar: if the backward convex map corresponds to
  $i: \un k \rat \un k'$ in $\unDelta$, we denote its image by $i\upperstar : X_{k'} 
  \to X_k$.  Similarly, the image of an ordinalic map, corresponding to
  $f: \un n \to \un k$ in $\unDelta$ is denoted $f\lowershriek : X_n \to X_k$.
  As observed in~\ref{iesq}, any identity-extension square in $\unDelta$
  \begin{equation}\label{eq:iesq}
    \vcenter{\xymatrix{
     \un a+\un n+\un b \ar[d]_{\id_a+f+\id_b=g}
                        &   \un n \ar@{ >->}[l]_-j\ar[d]^f  \\
     \un a+\un k+\un b  &   \un k \ar@{ >->}[l]^-i 
    }}
  \end{equation}
  is a pullback and hence a commutative square in $\newnabla$ between maps 
  from $\un a+\un n+\un b$ to $\un k$. The corresponding square of groupoids
  \begin{equation}\label{eq:BC}
 \vcenter{\xymatrix{
     X_{a+n+b} \ar[r]^-{j\upperstar }\ar[d]_{g\lowershriek}
     & X_{n} \ar[d]^{f\lowershriek} \\
      X_{a+k+b}\ar[r]_-{i\upperstar } &X_{k} .
  }}
  \end{equation}
therefore commutes by functoriality  (this is the `Beck--Chevalley condition' (BC).)

We say that $X$ satisfies the {\em iesq condition} if \eqref{eq:BC} not only commutes but is furthermore a pullback for every identity-extension square \eqref{eq:iesq}.
\end{blanko}

  If a nabla space $M: \newnabla \to \Grpd$ sends
  identity-extension squares to pullbacks then the composite
  $
  \simplexcategory\op\to \newnabla \to \Grpd
  $
  is a decomposition space.  This follows from the correspondence between
  iesq in $\unDelta$ and generic-free squares in $\simplexcategory$ 
  (Lemma~\ref{lem:iesq-genfree}).

A morphism of nabla spaces is called {\em \culf} if it is cartesian on 
(forward) ordinalic maps, i.e.~on arrows in $\un \simplexcategory\subset 
\newnabla$.  If $u: M'\Rightarrow M : \newnabla \to \Grpd$ is
a \culf natural transformation
  between functors that send identity-extension squares to pullbacks, then
  it induces a \culf functor between decomposition spaces.  Altogether:

\begin{prop}\label{prop:Nabla-to-Decomp}
	Precomposition with $\simplexcategory\op\to\newnabla$ defines
	a canonical functor
	$$
	\Fun^{\operatorname{culf}}_{\operatorname{iesq}}(\newnabla, \Grpd) \to 
	\kat{Decomp}^{\operatorname{culf}}
	$$
    from iesq (pseudo)-functors (and \culf
	(pseudo)-natural transformations) to decomposition spaces and \culf 
	functors.
\end{prop}

%%%%%%%%%%%%%%%%%%%%%%%%%%%%%%%%%%%%%%%%%%%%%%%%%%
\section{Sesquicartesian fibrations}
%%%%%%%%%%%%%%%%%%%%%%%%%%%%%%%%%%%%%%%%%%%%%%%%%%
\label{sec:sesq}

\begin{blanko}{Functors out of $\newnabla$.}
  In view of the Proposition~\ref{prop:Nabla-to-Decomp}, we are interested in defining functors
  out of $\newnabla$.  By its construction as a category of spans, this amounts to
  defining a covariant functor on $\unDelta$ and a contravariant functor on 
  $\unDelta_{\operatorname{convex}}$ which agree on objects, and such that for every
  pullback along a convex map the Beck--Chevalley condition holds.  Better
  still, we can describe these as certain fibrations over $\unDelta$, called
  sesquicartesian fibrations, which we now introduce.
\end{blanko}

\begin{blanko}{Sesquicartesian fibrations.}
  A {\em sesquicartesian fibration} is a cocartesian fibration
  $X\to\unDelta$ that is also cartesian over
  $\unDelta_{\operatorname{convex}}$, and in addition satisfies the
  Beck--Chevalley condition: for each pullback in $\unDelta$ of a convex
  map $\tau$,
  $$
  \xymatrix { \cdot \ar[r]^{\sigma'} \ar[d]_{\tau'} \drpullback & \cdot \ar[d]^\tau \\ \cdot 
  \ar[r]_\sigma & \cdot}
  $$
  the comparison map $\sigma'\lowershriek \tau'{}\upperstar \to \tau\upperstar \sigma\lowershriek$ is an isomorphism.

  Let $\kat{Sesq}$ be the category that has as objects the sesquicartesian
  fibrations and as arrows the functors of sesquicartesian fibrations
  (required to preserve cocartesian arrows and cartesian arrows over convex
  maps).
\end{blanko}

\begin{prop}\label{SesqNabla}
  There is a canonical functor
  $$
  \kat{Sesq}\longrightarrow \Fun(\newnabla,\kat{Cat}) .
  $$
\end{prop}
\noindent
Recall that $\Fun$ denotes the category of pseudo-functors 
and pseudo-natural transformations.

\begin{proof}
  Given a sesquicartesian fibration $p: X \to \unDelta$, 
  we can define a pseudo-functor
  $P:\newnabla\to \kat{Cat}$ as follows. On
  objects, send $\un n$ to the category $X_n$. 
  Send 
  a convex correspondence $\un n' \stackrel{j}{\lat} \un n
  \stackrel{f}\to \un k$ to the composite functor $X_{n'}
  \stackrel{j\upperstar} \to X_n \stackrel{f\lowershriek}
  \to X_k$.  Individually, the covariant and contravariant
  reindexing functors compose up to coherent isomorphisms because
  that's how cocartesian and cartesian fibrations work.  The
  Beck--Chevalley isomorphisms provide the coherence isomorphisms for
  general composition.

  On arrows: given a morphism $c: p'\to p$ of sesquicartesian 
  fibrations, assign a pseudo-natural transformation $u:P'\Rightarrow 
  P$: its component on $\un n$ is $c_n:{X}_n' \to 
  {X}_n$, its pseudo-naturality square on a backward convex 
  map $\un n' \stackrel{j}{\lat} \un n$ is given (at an object $x'\in 
  {X}'_{n'}$) by the isomorphisms $c(j\upperstar(x')) \simeq
  j\upperstar (c(x'))$ expressing that $c$ preserves cartesian arrows
  (but not chosen cartesian).  Similarly with the forward maps and 
  cocartesian lifts.  Again BC is invoked to ensure these are really 
  pseudo-natural.
\end{proof}

\begin{blanko}{Remark.}
  From work of Hermida~\cite{Hermida:repr-mult} and
  Dawson--Par\'e--Pronk~\cite{Dawson-Pare-Pronk:MR2116323}, it can be 
  expected that this functor is actually an equivalence, but we do not
  need this result and do not pursue the question further here.
\end{blanko}

\begin{blanko}{The iesq property.}\label{iesq-property}
  A sesquicartesian fibration $p:X\to\unDelta$ is said to have 
  the {\em iesq property} if for every identity-extension square 
  $$\xymatrix{
    \un a+\un n+\un b \ar[d]_{\id_a+f+\id_b=g} &
	\un n \ar@{ >->}[l]_-j\ar[d]^f & \\
    \un a+\un k+\un b  &   \un k \ar@{ >->}[l]^-i 
  }$$
  the diagram of categories
   $$\xymatrix{
     X_{a+n+b} \ar[r]^-{j\upperstar }\ar[d]_{g\lowershriek}
     & X_{n} \ar[d]^{f\lowershriek} \\
     X_{a+k+b}\ar[r]_-{i\upperstar } & X_{k}
  }$$
  not only commutes up to natural isomorphism (the BC condition),
  but is furthermore a homotopy pullback of categories (i.e.~it is equivalent to a 
  iso-comma square).

  Let $\kat{IesqSesq}$ be the category whose objects are the
  sesquicartesian fibrations $p:X\to\unDelta$ having the
  iesq property,
  and whose arrows are functors over $\unDelta$
  $$\xymatrix@C=2ex{
  X \ar[rd]_p \ar[rr]^c && Y \ar[ld]^q \\
  & \unDelta &
}$$
  that preserve cocartesian arrows and cartesian arrows (over convex maps),
  and satisfying the condition that for every arrow $f : \un n \to \un k$
  in $\unDelta$, the following square is a homotopy pullback:
  \begin{equation}\label{eq:subm}\xymatrix{ X_n \drpullback
  \ar[r]^{f\lowershriek} \ar[d]_c & X_{k} \ar[d]^{c} \\
  Y_n \ar[r]_{f\lowershriek} & Y_{k} .
  }\end{equation}
This condition on arrows $c:X\to Y$ is equivalent to saying that the associated 
  (pseudo)-natural transformation of pseudo-functors $\unDelta\to\kat{Cat}$ is homotopy
  cartesian, i.e.~all its (pseudo)-naturality squares are homotopy
  pullbacks.
\end{blanko}

\begin{prop}\label{prop:iesqsesqiNabla}
  The functor of Proposition \ref{SesqNabla} restricts to a functor
  $$
  \kat{IesqSesq} \longrightarrow 
  \Fun^{\operatorname{culf}}_{\operatorname{iesq}}(\newnabla,\kat{Cat})
  $$
\end{prop}
\noindent
Here $\Fun^{\operatorname{culf}}_{\operatorname{iesq}}(\newnabla,\kat{Cat})$ is
the subcategory of
$\Fun(\newnabla,\kat{Cat})$ whose objects are those
$X:\newnabla\to\kat{Cat}$ such that for every identity extension square the
corresponding Beck--Chevalley square is a homotopy pullback in \kat{Cat},
and whose morphisms are those pseudo-natural
transformations $X\to Y$ that are homotopy cartesian on (forward) ordinalic maps,
i.e.~on arrows in $\unDelta\subset\newnabla$.  Compare 
\ref{iesq-on-functors} for corresponding notions in
$\Fun(\newnabla,\Grpd)$.

%STRICT% \begin{blanko}{Strictly iesq sesquicartesian fibrations.}\label{sesquisplit}
%STRICT%   A sesquicartesian fibration is {\em split} when there are specified
%STRICT%   functorial cocartesian lifts for all maps and specified functorial
%STRICT%   cartesian lifts for convex maps, and such that the Beck--Chevalley
%STRICT%   isomorphisms are strict identities.  Strict morphisms of split
%STRICT%   sesquicartesian fibrations are morphisms that preserve the specified
%STRICT%   lifts.  A split sesquicartesian fibration is {\em strictly iesq} when the
%STRICT%   strictly commutative Beck--Chevalley squares are both homotopy pullbacks
%STRICT%   and strict pullbacks.  Let $\kat{IesqSesq}^{\operatorname{strict}}$ be
%STRICT%   the category of strictly iesq sequicartesian fibrations and the
%STRICT%   strict morphisms for which the square~\eqref{eq:subm} is both a strict 
%STRICT%   pullback and a homotopy pullback.
%STRICT% \end{blanko}
%STRICT% 
%STRICT% \begin{prop}
%STRICT%   The functor of Proposition \ref{SesqNabla} restricts to a functor
%STRICT%   $$
%STRICT%   \kat{IesqSesq}^{\operatorname{strict}} \longrightarrow 
%STRICT%   \Fun^{\operatorname{culf}}_{\operatorname{iesq}}(\newnabla,\kat{Cat})
%STRICT%   $$
%STRICT% \end{prop}
%STRICT% Here
%STRICT% $\Fun^{\operatorname{culf}}_{\operatorname{iesq}}(\newnabla,\kat{Cat})$ is
%STRICT% the subcategory of
%STRICT% $\Hom^{\operatorname{culf}}_{\operatorname{iesq}}(\newnabla,\kat{Cat})$
%STRICT% consisting of the strict functors and strict natural transformations.

Taking maximal subgroupoids to get a functor
$\Fun^{\operatorname{culf}}_{\operatorname{iesq}}(\newnabla,\kat{Cat})\to
\Fun^{\operatorname{culf}}_{\operatorname{iesq}}(\newnabla,\Grpd)$, and
combining Propositions~\ref{prop:iesqsesqiNabla} and
\ref{prop:Nabla-to-Decomp}, we obtain:

\begin{theorem}\label{thm:iesqsesqui->decomp}
  The constructions so far define a functor
  $$
  \kat{IesqSesq}
%STRICT%   ^{\operatorname{strict}} 
  \to 
  \kat{Decomp}^{\operatorname{culf}} .
  $$
\end{theorem}

\begin{blanko}{Decomposition categories.}
  The notion of decomposition space admits an obvious variation: that of
   {\em decomposition category} given by a (pseudo)-functor $X : 
  \simplexcategory\op\to\kat{Cat}$ such that the generic-free squares 
  are homotopy pullbacks.  It is clear that 
  iesq sesquicartesian fibrations define actually decomposition 
  categories---it is sort of artificial that we took groupoid interior
  as the last step to force the result to be a decomposition space instead 
  of a decomposition category (the motivation being of course to take
  homotopy cardinality and get coalgebras).  While for $\newnabla$-diagrams
  and $\simplexcategory\op$-diagrams it is obvious how to take groupoid 
  interior, corresponding to taking the left fibration associated to a 
  cocartesian fibration, this is not so for sesquicartesian fibrations, 
  which have genuinely categorical fibres.

  Decomposition categories arose also in our work \cite{GKT:MI} where the
  universal decomposition space of M\"obius intervals is in fact
  constructed as a decomposition $\infty$-category.  We leave for another
  occasion a more systematic study of decomposition categories.
\end{blanko}

\begin{blanko}{Example: monoids.}
  A monoid viewed as a  
  monoidal functor $X: (\unDelta, +,0) \to (\Grpd,\times,1)$
  defines a iesq sesquicartesian fibration.
  The contravariant functoriality on the convex maps is given as follows.  
  The cartesian lift of a convex map $\un a + \un n + \un b \lat \un n$ is simply the 
  projection
  $$
  X_{a+n+b} \simeq X_a \times X_n \times X_b \longrightarrow X_n ,
  $$
  where the first equivalence expresses that $X$ is monoidal.
  For any identity-extension square~\eqref{eq:iesq},
  it is clear that the corresponding diagram
  $$\xymatrix{
     X_{a+n+b}\drpullback \ar[r]^-{j\upperstar }\ar[d]_{g\lowershriek}
     & X_{n} \ar[d]^{f\lowershriek} \\
      X_{a+k+b}\ar[r]_-{i\upperstar } &X_{k}
  }$$
  is a pullback, since the upperstar functors are just projections.
  The associated decomposition space is the classifying space of the monoid.
\end{blanko}

%%%%%%%%%%%%%%%%%%%%%%%%%%%%%%%%%%%%%%%%%%%%%%%%%%
\section{From restriction species to iesq-sesqui}
%%%%%%%%%%%%%%%%%%%%%%%%%%%%%%%%%%%%%%%%%%%%%%%%%%
\label{sec:RSp->iesq}

In order to construct nabla spaces satisfying the iesq property, we can
construct sesquicartesian fibrations satisfying iesq, and then take maximal
sub-groupoid.

All our examples originate as the left leg of a two-sided fibration,
as we proceed to explain.

\begin{blanko}{Two-sided fibrations.}
  Classically (the notion is due to Street~\cite{Street:LNM420}), 
  a {\em two-sided fibration} is a span of functors
  $$
  \xymatrix{X \ar[d]_p\ar[r]^q & T \\
  S&
  }$$
  such that
 
  ---
  $p$ is a cocartesian fibration whose
  $p$-cocartesian arrows are precisely the $q$-vertical arrows, 
 
  ---
  $q$ is a cartesian fibration whose
  $q$-cartesian arrows are precisely the $p$-vertical arrows,
  
  --- for $x\in X$, an arrow $f: px\to s$ in $S$ and $g:t \to qx$ in $T$, the
  canonical map 
  $f\lowershriek g\upperstar  x \to g\upperstar f \lowershriek x$ is an 
  isomorphism.
  
  In the setting of $\infty$-categories, Lurie~\cite[\S2.4.7]{Lurie:HA}
  (using the terminology `bifibration') characterises two-sided fibrations as
  functors $X \to S \times T$ subject to a certain horn-filling condition,
  which among other technical advantages makes it clear that the notion is
  stable under pullback along functors $S' \times T' \to S \times T$.
  The classical axioms are derived from the horn-filling condition.
\end{blanko}

\begin{blanko}{Comma categories.} 
  $\Ar(\CC) \xrightarrow{(\mathrm{codom}, \mathrm{dom})} \CC\times\CC$
  is a two-sided fibration.
  Given categories and functors 
  $$\xymatrix{
  & S \ar[d]^G \\
  T \ar[r]_F & I
  }$$
  the {\em comma category} $T \comma S$
  is the category whose objects are triples $(t,s,\phi)$, where $t\in T$, $s\in 
  S$,
  and $\phi:Ft\to Gs$.  More formally it is defined as 
  the pullback two-sided fibration
  $$\xymatrix{
    T \comma S \drpullback  \ar[r]\ar[d] & \Ar(I) 
    \ar[d]^{(\mathrm{codom},\mathrm{dom})} \\
     S \times T \ar[r]_{G\times F} & I \times I .
  }$$ 
  Note that the factors come in the opposite order: $T\comma S \to S$ is the
  cocartesian fibration, and $T\comma S \to T$ the cartesian fibration.
  The left leg cocartesian fibration comes with a canonical splitting.
  The two-sided fibration sits in a comma square which we depict like this:
  $$\xymatrix{
  T\comma S \ar[r] \ar[d] \ar@{}[rd]|{\Leftarrow} & T \ar[d] \\
  S \ar[r] & I
  }$$
\end{blanko}

\begin{lemma}\label{sesquilemma}
  In a diagram
  $$
  \xymatrix{X \times_T R \ar@/_1.5pc/[dd]_f \drpullback \ar[r] \ar[d] & R \ar[d]^w\\
  X \ar[d]^p\ar[r]_q & T \\
  \un\simplexcategory&
  }$$
  where 
  
  --- $(p,q):X \to \un\simplexcategory \times T$ is a two-sided fibration;
  
---  $p:X \to \un\simplexcategory$ is a iesq sesquicartesian fibration; and
  
  ---  $w:R \to T$ is a cartesian fibration; 
  
  \noindent we have
  \begin{enumerate}
  	\item $f$ is a iesq sesquicartesian fibration. 
  
%STRICT%   	\item if the sequicartesian fibration $X \to \un\simplexcategory$ is 
%STRICT% 	strictly iesq (in the sense of \ref{sesquisplit}) and if the cartesian 
%STRICT% 	fibration $R \to T$ is split, then the sesquicartesian fibration $f$ is 
%STRICT% 	again strictly iesq.  NOT YET VERIFIED.
  
        \item the map $X \times_T R \to X$
          is a morphism of iesq sesquicartesian fibrations from $f$ to $p$ (in the sense of 
  \ref{iesq-property}).%is a submersion.
  \end{enumerate}

\end{lemma}

\begin{proof}
  (1) $f$ is a cocartesian fibration because it is the left leg of the pullback
  two-sided fibration of $X \to \unDelta\times T$ along $\unDelta
  \times R \to \unDelta \times T$.
%STRICT%   (1s): This takes care of the splitting too.
  The $f$-cartesian lift of a given convex arrow has components $(\ell, c)$
  where $\ell$ is a $p$-cartesian lift to $X$, and $c$ is a $w$-cartesian
  lift of $q(\ell)$.
%STRICT%   (2s): the components of the cartesian splitting come from the cartesian
%STRICT%   splittings of $p$ and $w$.  (The cocartesian splittings were provided in (1s).)
  Given the pullback square 
  $$\xymatrix{
     X_{a+n+b}\drpullback \ar[r]^-{j\upperstar }\ar[d]_{\sigma'\lowershriek}
     & X_{n} \ar[d]^{\sigma\lowershriek} \\
      X_{a+k+b}\ar[r]_-{i\upperstar } &X_{k}
  }$$
  expressing that $X \to \unDelta$ has the iesq property, the corresponding square 
  for  $X\times_T R \to \unDelta$ is simply obtained applying
  $- \times_T R$ to it, hence is again a pullback, so $f$ has the 
  iesq property.

  %(2):
  (2) By construction $X \times_T R \to X$ preserves cocartesian arrows and cartesian arrows over convex maps, so it is indeed a morphism of sesquicartesian fibrations.
  For each arrow $\sigma: n \to k$ in $\unDelta$, the square required 
  to be a pullback is
  $$\xymatrix{
    X_n \times_T R \drpullback \ar[r]^{\sigma\lowershriek \times_T R} \ar[d] & X_{k}\times_T R 
	\ar[d] \\
	X_n \ar[r]_{\sigma\lowershriek} & X_{k}
  }$$
  which is clear.
\end{proof}

\begin{blanko}{Restriction species and directed restriction species.}
  Recall that $\I$ denotes the category of finite sets and injections,
  and that a restriction species is a functor $R: \I\op\to\Grpd$,
  or equivalently, a 
%STRICT%   split 
  right fibration $\R\to\I$.  Recall also that $\C$ denotes the
  category of finite posets and convex maps, and that
  a {\em directed restriction
  species} is a functor $R : \C\op\to\Grpd$, or equivalently, a
%STRICT%   split 
  right fibration $\R\to\C$.
  
  We are going to establish that every ordinary restriction species
  and every directed restriction species defines naturally a
%STRICT%   strictly 
  iesq
  sesquicartesian fibration.  We will do the proofs for directed restriction
  species, and then exploit the fact that ordinary restriction species
  are a special kind of directed restriction species to deduce the 
  results also for ordinary restriction species.
\end{blanko}

\begin{prop}\label{CDelta}
  The projection $\C \comma \un\simplexcategory \to \un\simplexcategory$ 
  is a
%STRICT%   strictly
  iesq sesquicartesian fibration.
\end{prop}

\begin{proof}
  The comma category is taken over $\kat{Poset}$.  The objects of $\C\comma
  \un\simplexcategory$ are poset maps $P \to \un k$, and the arrows are
  squares in $\kat{Poset}$
  $$\xymatrix{
     Q \ar[r]\ar[d] & P \ar[d] \\
     \un n \ar[r] & \un k
  }$$
  with $Q\to P$ a convex map and $\un n \to \un k$ a monotone
  map.  Just from being a comma category projection, $\C \comma
  \un\simplexcategory \to \un\simplexcategory$ is a (split)
  cocartesian
  fibration.  The chosen cocartesian arrows are squares in $\kat{Poset}$ of
  the form
  $$\xymatrix{
     P \ar[r]^=\ar[d] & P \ar[d] \\
     \un n \ar[r] & \un k  .
  }$$
  Over $\un\simplexcategory_{\text{convex}}$ it is also a
  (split) cartesian fibration, as follows readily from Lemma \ref{pbkconv} on pullback stability of 
  convex maps in $\kat{Poset}$: the 
  cartesian arrows over a convex map are squares in $\kat{Poset}$ of the form
  $$\xymatrix{
     P\drpullback \ar@{ >->}[r] \ar[d]_{i\upperstar \beta} & P' \ar[d]^\beta \\
     \un k \ar@{ >->}[r]_i & \un k'  .
  }$$
  The chosen cartesian arrows are the squares in which the map $P \to P'$
  is an actual inclusion.

  Finally for the iesq property,
  we need to check that given
      $$\xymatrix{
	  \un a+\un n+\un b=\un n' 
	    \ar[d]_{\id_a+f+\id_b}^{=g} &   \un n \ar@{ >->}[l]_-j\ar[d]^f & \\
	\un a+\un k+\un b =\un k' &   \un k \ar@{ >->}[l]^-i 
	  }$$
  the resulting strictly commutative square
  $$\xymatrix{
     \C_{/\un n'}\drpullback \ar[r]^-{j\upperstar }\ar[d]_{g\lowershriek}
     & \C_{/\un n} \ar[d]^{f\lowershriek} \\
      \C_{/\un k'}\ar[r]_-{i\upperstar } &\C_{/\un k}
  }$$
  is a pullback.  To this end, note first that lowershriek
  functors between slices are cartesian fibrations, so it is enough to 
  show that this square is a strict pullback.  We first compute the strict
  pullback at the level of objects.
  A pair $(P'\xrightarrow\beta\un k',P\xrightarrow\alpha\un n)$ lies in 
  the pullback $\C_{/\un k'}\times_{\C_{/\un k}} \C_{/\un n}$  
  if $i\upperstar\beta=f\lowershriek\alpha$, that is, $P$ is an actual 
  subposet of
  $P'$ and this diagram is a pullback:
  $$\xymatrix@R=3ex{
  P \ar[r]
  \ar[d]_\alpha \drpullback & P' \ar[dd]^\beta \\
  \un n \ar[d]_f  & \\
  \un k \ar@{ >->}[r]_i & \un k'   .
  }$$
  The claim is then that there is a unique way to 
  complete this diagram to
  $$\xymatrix{
  P \ar[r]
  \ar[d]_(0.5)\alpha \drpullback & P' \ar[dd]_(0.25)\beta|\hole \ar@{-->}[rd] &\\
  \un n \ar@{ >->}[rr]\ar[d]  && \un n' \ar[ld] \\
  \un k \ar@{ >->}[r] & \un k'  .
  }$$
  Indeed, at the level of elements, $P'$ is constituted by three subsets,
  namely the inverse images $P'_{\un a}$, $P'_{\un k}$ and $P'_{\un b}$.
  (We don't need to worry about the poset structure, since we already know
  all of $P'$.  The point is that the covariant functoriality does not
  change the total space.)  We now define $P'\to \un n'=\un a+\un n+\un b$
  as follows: we use $\beta$ to define $P'_{\un a} \to \un a$ and $P'_{\un
  b} \to \un b$ on the outer subsets, and on the middle subset we use
  $\alpha$ to define $P'_{\un k} = P\to \un n$.  Conversely, an
  element in $\C_{/\un n'}$ defines a element in the pullback, and
  it is clear that the two constructions are inverse to each other.  Having
  established that the two groupoids have the same objects, it remains to
  check that their automorphism groups agree.  An automorphism of a pair
  $(P'\xrightarrow\beta\un k',P\xrightarrow\alpha\un n)$ is an automorphism
  of $P'$ compatible with the $k'$-layering and whose restriction to $\un k$
  is furthermore compatible with the refined layering here, given by $P 
  \to \un n$.  But this is precisely to say that it is an
  automorphism of $P'$ that is compatible with the layering $P'\to \un
  n'$ constructed.
\end{proof}

\begin{prop}\label{prop:DRSp->iesq}
  There is a natural functor
  $$
  \kat{DRSp} \simeq\kat{RFib}_{/\C} \to \kat{IesqSesq} ,
  $$
  which takes a directed restriction species $R:\C\op\to\Grpd$ with
  associated 
%STRICT%   split 
  right fibration $\R\to\C$ to the comma category
  projection $\R \comma \un\simplexcategory \to \un\simplexcategory$.
\end{prop}

\begin{proof}
  Just note that stacking pullbacks on top of a comma square yields 
  again comma squares:
  $$\xymatrix{
  \R'\comma \unDelta \drpullback \ar[r] \ar[d] & \R' \ar[d] \\
  \R\comma \unDelta \drpullback \ar[r] \ar[d] & \R \ar[d] \\
  \C \comma \unDelta \ar[r]\ar[d] \ar@{}[rd]|{\Leftarrow} & \C \ar[d] \\
  \unDelta \ar[r] & \kat{Poset} .
  }$$
  Now $\C\comma\unDelta\to\unDelta$ is a iesq sesquicartesian fibration by
  Proposition~\ref{CDelta}, so the statement about objects follows
  from Lemma~\ref{sesquilemma}~(1) and the statement about morphisms
  from Lemma~\ref{sesquilemma}~(2).
\end{proof}

From these results for directed restriction species, the analogous 
results for ordinary restriction species can be deduced,
remembering from \ref{RSpsubDRSp} that $\I\to\C$ is a 
%STRICT% (split) 
right fibration.

\begin{cor}\label{IDelta}
  The projection $\I \comma \un\simplexcategory \to 
  \un\simplexcategory$ is a
%STRICT%   strictly 
  iesq sesquicartesian fibration.
\end{cor}

\begin{cor}\label{cor:RSp->iesq}
  For any ordinary restriction species $R:\I\op\to\Grpd$ with
  associated 
%STRICT%   split 
  right fibration $\R\to\I$, the comma category
  projection $\R \comma \un\simplexcategory \to \un\simplexcategory$
  is a
%STRICT%   strictly 
  iesq sesquicartesian fibration.
\end{cor} 

Proposition~\ref{prop:DRSp->iesq}, together with
Theorem~\ref{thm:iesqsesqui->decomp} (that is,
Propositions~\ref{prop:Nabla-to-Decomp} and \ref{prop:iesqsesqiNabla}),
gives the following result, summarising our constructions so far.

\begin{theorem}\label{thm:RSP&DRSp->DS}
  The constructions above define functors
  $$
  \kat{RSp} 
  \xrightarrow{\ref{RSpsubDRSp}} 
  \kat{DRSp} 
  \xrightarrow{\ref{prop:DRSp->iesq}} 
  \kat{IesqSesq}
  \xrightarrow{ \ref{prop:iesqsesqiNabla}} 
  \Fun_{\operatorname{iesq}}^{\operatorname{culf}}(\newnabla,\Grpd)
  \xrightarrow{\ref{prop:Nabla-to-Decomp}} 
  \kat{Decomp}^{\mathrm{culf}} .
  $$
\end{theorem}
These functors are not exactly fully faithful but we shall see in the next 
section that they become fully faithful when suitably sliced.

\begin{blanko}{Unpacking, and comparison with the discussion in \S\ref{sec:DRSp}.}\label{unsubmersive} 
  Given a directed restriction species $R:\C\op\to\Grpd$, we may consider
  the associated right fibration $p:\R\to\C$ as a morphism in
  $\kat{RFib}_{/\C}$ from $p$ to the terminal object $\C\to\C$.
  Theorem~\ref{thm:RSP&DRSp->DS} then associates to this a decomposition
  space $\ds R:\simplexcategory\op\to\Grpd$ with a \culf functor
  $\Psi(p):\ds R\to\ds C$, constructed via iesq-sesqui and nabla spaces.

  Indeed, we have a functor
  $$
  \Psi:\kat{RFib}_{/\C}\to\kat{Decomp}_{/\ds C}^{\operatorname{culf}}
  $$
  to the category of decomposition spaces which are \culf over $\ds C$.
   
  Let us unpack the constructions.  Consider the pullback of $p$ to the
  comma categories
  $$
  \xymatrix{ \R_{/\un n}\rto\dto\drpullback & 
  \R\comma \unDelta \drpullback \ar[r] \ar[d] & \R \ar[d]^p \\
  \C_{/\un n}\rto \drpullback \ar[d] &  \C \comma \unDelta \ar[r]\ar[d] 
  \ar@{}[rd]|{\Leftarrow}& \C \ar[d]\\
  1 \ar[r]_{\name{\un n}} &  \unDelta  \ar[r] & \kat{Poset} .
  }$$
  The values of the simplicial groupoids $\ds R$ and $\ds C$ at $[n]$, are
  groupoid interiors of the fibres over $\un n\in\un\simplexcategory$,
  $$
  \ds C_n=( \C\comma \unDelta)^\iso_{\un n}=\C^\iso_{/\un n},\qquad
  \ds R_n=( \R\comma \unDelta)^\iso_{\un n}=\R^\iso_{/\un n}
  =\C^\iso_{/\un n}\times_{\C^\iso}\R^\iso  ,
  $$
  and $\Psi(p)_n:\ds R_n\to \ds C_n$ is the canonical projection.
  The simplicial structure is given as follows:
  \begin{itemize}
 
	\item A generic map $g:[n]\genmap[k]$ in $\simplexcategory$ and 
  the corresponding $\un g:\un k\to \un n$ in $\un\simplexcategory$
  induce, by postcomposition, the map of groupoids 
  $$
  \ds C_k\to\ds C_n
  ,\qquad(P\to\un k)\mapsto\un g\lowershriek(P\to \un k)=(P\to\un k\to \un n).
  $$
  This in turn induces the map $\ds R_k\to\ds R_n$,
  $$\xymatrix{
  \ds R_k\drpullback\rto\dto&\ds R_n\drpullback\rto\dto&\ds R_1\dto\\
  \ds C_k           \rto    &\ds C_n           \rto    &\ds C_1.}
  $$
  and hence the projection $\ds R\to \ds C$ is cartesian on generic maps.

  \item A free map $f:[n]\rat[k]$ in $\simplexcategory$ and the associated 
  convex map $\un f:\un n\rat\un k$ in $\un\simplexcategory$ induce, 
  by pullback, the homomorphism
  $$
  \ds C_{k}\to\ds C_n,\qquad(P\to\un{k})\mapsto(\un f\upperstar P\to \un n).
  $$
  The definition of $\ds R_k\to\ds R_n$
  uses the directed restriction species structure,
  $$
  \ds C_k\times_{\ds C_1}\ds R_1\longrightarrow
  \ds C_n\times_{\ds C_1}\ds R_1,\qquad
  (P\to\un k,\;S)\longmapsto
  (f\upperstar P\to\un n,\,(S|f\upperstar P)\,) .
  $$
  \end{itemize}
\end{blanko}

%%%%%%%%%%%%%%%%%%%%%%%%%%%%%%%%%%%%%%%%%%%%%%%%%%
\section{Decalage and fully faithfulness}
%%%%%%%%%%%%%%%%%%%%%%%%%%%%%%%%%%%%%%%%%%%%%%%%%%

We have already exploited (Proposition \ref{prop:Dec(C)}) the decalage formulae 
  $$
  \Dec_\bot(\ds C) \simeq \fatnerve \C^\low
  \qquad
  \Dec_\top \ds C \simeq \fatnerve (\C^\upper)\op
  $$
  which we now generalise as follows.  For each directed restriction
  species $R$, we can pull back the corresponding right fibration $\R\to\C$
  to these subcategories of upper- and lower-set inclusions, giving
  $$
  \R^\low := \C^\low \times_{\C} \R, \qquad \qquad 
  \R^\upper := \C^\upper \times_{\C} \R, 
  $$
  the categories of $R$-structures and their lower-set and upper-set
  inclusions.  Thus we have pullback functors
  $$
  \kat{RFib}_{/\C^\low}    
  \xleftarrow{\operatorname{pbk}}\kat{RFib}_{/\C}\xrightarrow{\operatorname{pbk}}
  \kat{RFib}_{/\C^\upper} .
  $$

\begin{prop}\label{prop:Dec-of-DRSp}
  We have the following natural (levelwise)
  equivalences of simplicial groupoids:
  $$
  \Dec_\bot \ds R \simeq \fatnerve \R^\low 
  \qquad
  \qquad
  \Dec_\top \ds R \simeq \fatnerve (\R^\upper)\op    .
  $$
\end{prop}

\begin{proof}
  The equivalences are expressed by commutativity of the left-hand faces
  (incident with the edge labelled by the functor $\Psi:\R\mapsto\ds R$) of
  the cube in the following lemma.
\end{proof}

\begin{lemma}\label{lem:cube}
  We have the commutative diagram
  $$\xymatrix @!=15pt {
  &&\kat{RFib}_{/\C^\upper} \ar[d]^{\fatnerve} \ar[rrrrd]&&&& \\
  \kat{RFib}_{/\C}\ar[ddd]_\Psi\ar[rru]\ar[rrrrd]&&  {\scriptsize 
  \kat{RFib}_{/\fatnerve\C^\upper}} \ar@{=}[d]|(0.5)\hole &&&&\kat{Fib}_{/\C^\iso} 
  \ar[ddd]^{\fatnerve}\\
  &&  {\scriptsize \kat{RFib}_{/(\Dec^\top \!\ds C) \op}} \ar@{=}[d]    &&\kat{RFib}_{/\C^\low}\ar[dd]^{\fatnerve}\ar[rru]&&\\
  &&\kat{LFib}_{/\Dec_\top \!\ds C} \ar[rrrrd]|(0.5)\hole &&&&\\
  \kat{Decomp}_{/\ds C}^{\operatorname{culf}} 
  \ar[rru]^{\Dec_\top}\ar[rrrrd]_{\Dec_\bot}&&&&{\scriptsize\kat{RFib}_{/\fatnerve{\C^\low}}}\ar@{=}[d]&&\kat{Cart}_{/\ds C_1}\\
  &&&&\kat{RFib}_{/\Dec_\bot \!\ds C} \ar[rru]&&
  }$$
\end{lemma}

\begin{proof}
  We first prove that the left-hand faces commute.
  In simplicial degree zero the images of $p:\R\to\C$ clearly coincide: they are
  $p^\iso:\ds R_1=\R^\iso\to \ds C_1=\C^\iso$.  Analogously to
  \ref{unsubmersive} 
  we can write
 \begin{align*}
   (\Dec_\bot\ds R)_k&=\ds R_{k+1}
     =\ds C_{k+1}\times_{\ds C_1}\ds R_1
  =(\Dec_\bot\ds C)_{k}\times_{\ds C_1}\ds R_1
  \\
   (\fatnerve \R^\low)_k&=(\fatnerve \C^\low)_k\times_{(\fatnerve \C^\low)_0}(\fatnerve \R^\low)_0
               =(\fatnerve \C^\low)_k\times_{\ds C_1}\ds R_1  ,
  \end{align*}
  and similarly for $\Dec_\top$ and the categories of upper-set
  inclusions.  From Proposition~\ref{prop:Dec(C)} we have canonical equivalences
  of simplicial groupoids $\Dec_\bot\ds C=\fatnerve \C^\low$ and
  $\Dec_\top\ds C=\fatnerve (\C^\upper)\op$.  We also have commuting
  diagrams for generic or bottom face maps
  $$\xymatrix{
  (\Dec\ds R)_{k}
  \rto^-=\dto&  \Dec(\ds C)_{k}\times_{\ds C_1}\ds R_1
  \dto\rto^-\cong & (\fatnerve \C^\low)_k\times_{\ds C_1}\ds R_1\rto^-=\dto&
  (\fatnerve \R^\low )_k\dto
  \\
  (\Dec\ds R)_{n}
  \rto^- =&\Dec(\ds C)_{n}\times_{\ds C_1}\ds R_1  
  \rto^-\cong & (\fatnerve \C^\low )_n\times_{\ds C_1}\ds R_1\rto^-=&(\fatnerve \R^\low )_n
  }$$
  The diagram for $d_\top:[k{-}1]\to[k]$ also commutes:
  $$
  \xymatrix{
  \left({\begin{array}{c}\scriptstyle 
  P\\[-1.5mm]\downarrow\\[-1.6mm]\scriptstyle\un{k{+}1}\end{array}}
  , S \right)\qquad
  \ar@{<->}[r]\ar@{|->}[d]&\ar@{|->}[d] 
  \left(P_{\un 1}\subseteq P_{\un2}\subseteq\dots\subseteq P, S 
  \right)\\
    \left({\begin{array}{c}\scriptstyle {\un 
	d^\top}^{\!*\!\!}P\\[-1.5mm]\downarrow\\[-1.6mm]\scriptstyle\un{k}\end{array}}
    , (S|{\un d^\top}\upperstar \!P)\right)
  \ar@{<->}[r] &   \left( P_{\un1}\subseteq\dots\subseteq P_{\un k},\, 
  (S|P_{\un k}) \right). 
  }
  $$
  This shows that the two left-hand faces commute.
  
  The top face is just pullback to $\C^\iso$ taken in two steps in two
  ways.  For the bottom face, observe first that $\ds C_1$ is the constant
  simplicial groupoid with value $\ds C_1 = \C^\iso$.  The bottom face
  commutes because both ways around send a \culf map $\ds R \to \ds C$ to
  the (obviously cartesian) simplicial map of constant simplicial groupoids
  $\ds R_1 \to \ds C_1$.  The right-hand faces are easier to understand
  with $\kat{RFib}_{\fatnerve \C^\upper}$ instead of
  $\kat{LFib}_{/\Dec_\top \ds C}$ and $\kat{RFib}_{\fatnerve \C^\low}$
  instead of $\kat{RFib}_{/\Dec_\bot \ds C}$: commutativity of the two
  squares then just amounts to the fact that the fat nerve commutes with
  pullbacks.
\end{proof}

Since ordinary restriction species are just directed restriction 
species supported on discrete posets, Proposition~\ref{prop:Dec-of-DRSp}
implies the following
result, remembering that for discrete posets, every inclusion is
both a lower-set and an upper-set inclusion:
\begin{cor}\label{cor:DecR=NR}
  For an ordinary restriction species $\R\to\I$ with associated decomposition space 
  $\ds R$, we have
  $$
  \Dec_\bot(\ds R) \simeq \fatnerve \R \qquad
  \Dec_\top(\ds R) \simeq \fatnerve \R\op .
  $$
\end{cor}

\begin{theorem}\label{thm:DRSp->decomp/C=ff}
   The functor
   $$
   \Psi: \kat{DRSp} \longrightarrow \kat{Decomp}^{\operatorname{culf}}_{/\ds C}
   $$
   is fully faithful.
\end{theorem}
\begin{proof}
  From the cube diagram in Lemma~\ref{lem:cube} we get the commutative
  square
    $$\xymatrix{
  \kat{RFib}_{/\C} \ar[rr]^-{(\mathrm{pbk},\mathrm{pbk})} \ar[d]_-\Psi 
  &&
  \kat{RFib}_{/\C^\low} \!\!\underset{\kat{Fib}_{/\C^\iso}}\times \!\!
  \kat{RFib}_{/\C^\upper} \ar[d]^-{\fatnerve \underset{\fatnerve}\times \fatnerve }
   \\
    \kat{Decomp}^{\operatorname{culf}}_{/\ds C}  \ar[rr]_-{(\Dec_\bot, \Dec_\top)}  
  &&
    \kat{RFib}_{/\Dec_\bot \ds C} \!\underset{\kat{Cart}_{/\ds C_1}}\times 
	\!
  \kat{LFib}_{/\Dec_\top \ds C}   \; .
  }$$
  Now the main point is that
  the pair of pullback functors is
  jointly fully faithful.  Indeed, a transformation 
  is natural in all convex maps if and only if it is natural
  in both lower-set 
  inclusions and upper-set inclusions, since every convex inclusion
  factors (non-uniquely) as a lower-set inclusion followed by an 
  upper-set inclusion.  Since also the fibre product of fat nerves is 
  fully faithful, and since the pair of
  Decs is faithful, we conclude that $\Psi$ is fully faithful.
\end{proof}

\begin{thm}\label{thm:RSp->decomp/I=ff}
  The functor 
  $$
  \kat{RSp} \longrightarrow \kat{Decomp}_{/\ds I}^{\operatorname{culf}} 
  $$
  is fully faithful.
\end{thm}

\begin{proof}
  In the commutative diagram
$$\xymatrix{
\kat{RSp} \ar[r] \ar[d]_{\operatorname{f.f.}} &
 \kat{Decomp}^{\operatorname{culf}}_{/\ds I} \ar[d]^{\operatorname{f.f.}} \\
\kat{DRSp} \ar[r]_-{\operatorname{f.f.}} & \kat{Decomp}^{\operatorname{culf}}_{/\ds C}
}$$
$\kat{RSp} \subset \kat{DRSp}$ is clearly fully faithful; 
$\kat{DRSp} \to
\kat{Decomp}^{\operatorname{culf}}_{/\ds C} $ is fully faithful by 
Theorem~\ref{thm:DRSp->decomp/C=ff}, and 
$\kat{Decomp}^{\operatorname{culf}}_{/\ds I}\to 
	\kat{Decomp}^{\operatorname{culf}}_{/\ds C}$
	is fully faithful since 
	$\ds I \to \ds C$ is a monomorphism in 
  $\kat{Decomp}^{\operatorname{culf}}$.
\end{proof}

%%%%%%%%%%%%%%%%%%%%%%%%%%%%%%%%%%%%%%%%%%%%%%%%%%
\section{Remarks on strictness}
%%%%%%%%%%%%%%%%%%%%%%%%%%%%%%%%%%%%%%%%%%%%%%%%%%

\label{sec:strict}

Since our general philosophy is that the homotopy content
is the essence---and in the end we want to take homotopy cardinality 
anyway---we have worked in this paper with groupoids up to homotopy:
when we say simplicial groupoid, we mean pseudo-functor 
$\simplexcategory\op\to \Grpd$, and all pullbacks mentioned are
homotopy pullbacks.

Nevertheless, one may rightly feel that it is nicer to work with strict
simplicial objects.  In the present situation one can actually have a
strict version of everything, if just restriction species and
directed restriction species are assumed to be {\em strict} groupoid-valued
functors, not pseudo-functors (and their morphisms {\em strict} natural 
transformations 
rather than pseudo-natural transformations).  It is doable to trace through all the
construction with sufficient care to ensure that the resulting
decomposition spaces are again strict.

We finish the paper by outlining the arguments going into this.  First of
all: 

\begin{blanko}{Strict decomposition spaces.}
  We define {\em strict decomposition spaces} to be strict functors
  $\simplexcategory\op\to\Grpd$ such that the generic-free squares are
  simultaneously strict pullbacks and homotopy pullbacks.
\end{blanko}
Note that the squares in question are already strictly commutative since
they are strict simplicial identities, so in practice the pullback
condition happens because it is a strict pullback in which one of the legs
is an iso-fibration.

For example, the fat nerve of a small category is a strict decomposition
space: it is clearly a strict functor, the Segal squares are readily seen
to be strict pullbacks, and the face maps
are iso-fibrations because the coface
maps in $\simplexcategory$ are injective on objects.

\begin{blanko}{Strict \culf functors.}
  We define a {\em strict \culf functor} to be a strictly simplicial map, whose
  naturality squares on generic maps are simultaneously strict pullbacks
  and homotopy pullbacks.
\end{blanko}
Again, this typically happens when the simplicial map is degree-wise an
iso-fibration.

\begin{theorem}
  The functors $\kat{RSp} \to \kat{DRSp} \to \kat{Decomp}^{\mathrm{culf}}$ of
  Theorem~\ref{thm:RSP&DRSp->DS} take strict (directed) restriction species and
  their strict morphisms to strict decomposition spaces and strict \culf
  functors.
\end{theorem}

Let us explain the main intermediate step.

\begin{blanko}{Strictly iesq sesquicartesian fibrations.}
  A sesquicartesian fibration is {\em split} when there are specified
  functorial cocartesian lifts for all maps and specified functorial
  cartesian lifts for convex maps, and such that the Beck--Chevalley
  isomorphisms are strict identities.  A split sesquicartesian fibration is
  {\em strictly iesq} when the strictly commutative Beck--Chevalley squares
  are both strict pullbacks and homotopy pullbacks.  A strict morphism of
  strictly iesq sesquicartesian fibrations is by definition a functor that
  preserves the specified lifts, both cocartesian and cartesian, and for
  which the square~\eqref{eq:subm} is both a strict pullback and a homotopy
  pullback.
\end{blanko}

\begin{lemma}
  The functors $\kat{RSp} \to \kat{DRSp} \to \kat{IesqSesq}$ of
  Proposition~\ref{prop:DRSp->iesq} take strict (directed) restriction species
  and their strict morphisms to strictly iesq sesquicartesian fibrations and
  strict morphisms.
\end{lemma}

The main ingredient in checking this is the fact that the base case 
$\C \comma \unDelta 
\to \unDelta$ is a strictly iesq sesquicartesian fibration.  This follows
from inspection of the proof of Proposition~\ref{CDelta}, where in fact
the crucial pullback square was established as a strict pullback along an 
iso-fibration.  For this we exploited in particular that the pullbacks
of convex maps can be taken to be actual subset inclusions.

For the general strict directed restriction species (which includes $\I$),
the proof follows from niceness of comma categories, including the
fact that comma-category projections are always split cartesian and 
cocartesian fibrations, 
and therefore the top squares in the proof of Proposition~\ref{prop:DRSp->iesq}
can be taken to be strict pullbacks.

\bigskip

Finally, it is straightforward to verify that all the strictnesses are preserved 
by the functor of Proposition~\ref{prop:iesqsesqiNabla}
to (suitably strict) nabla spaces, and from there to strict decomposition spaces
via Proposition~\ref{prop:Nabla-to-Decomp}.

\bigskip

We stress that for the sake of taking homotopy cardinality to obtain
incidence coalgebras, the strictness is irrelevant.

% \bibliographystyle{../scplain}
% \bibliography{../GKT1}

\begin{thebibliography}{10}

\bibitem{Aguiar-Bergeron-Sottile}
{\sc Marcelo Aguiar, Nantel Bergeron, {\rm and }Frank Sottile}.
\newblock {\em Combinatorial {H}opf algebras and generalized
  {D}ehn-{S}ommerville relations}.
\newblock Compos. Math. {\bf 142} (2006), 1--30.
\newblock \arxiv{math/0310016}.

\bibitem{Aguiar-Mahajan}
{\sc Marcelo Aguiar {\rm and }Swapneel Mahajan}.
\newblock {\em Monoidal functors, species and {H}opf algebras}, vol.~29 of CRM
  Monograph Series.
\newblock American Mathematical Society, Providence, RI, 2010.
\newblock With forewords by Kenneth Brown, Stephen Chase, and Andr{\'e}
  Joyal.

\bibitem{Baez-Dolan:finset-feynman}
{\sc John~C. Baez {\rm and }James Dolan}.
\newblock {\em From finite sets to {F}eynman diagrams}.
\newblock In {\em Mathematics unlimited---2001 and beyond}, pp. 29--50.
  Springer, Berlin, 2001.

\bibitem{Bergeron-Labelle-Leroux}
{\sc Fran{\c{c}}ois Bergeron, Gilbert Labelle, {\rm and }Pierre Leroux}.
\newblock {\em Combinatorial species and tree-like structures}, vol.~67 of
  Encyclopedia of Mathematics and its Applications.
\newblock Cambridge University Press, Cambridge, 1998.
\newblock Translated from the 1994 French original by Margaret Readdy, with a
  foreword by Gian-Carlo Rota.

\bibitem{Bergner-et.al:1609.02853}
{\sc Julia~E. Bergner, Ang{\'e}lica~M. Osorno, Viktoriya Ozornova, Martina
  Rovelli, {\rm and }Claudia~I. Scheimbauer}.
\newblock {\em 2-{S}egal sets and the {W}aldhausen construction}.
\newblock To appear in Topol. Appl. (2017).
\newblock \arxiv{1609.02853}.

\bibitem{Cartier-Foata}
{\sc Pierre Cartier {\rm and }Dominique Foata}.
\newblock {\em Probl{\`e}mes combinatoires de commutation et
  r{\'e}arrangements}.
\newblock No.~85 in Lecture Notes in Mathematics. Springer-Verlag, Berlin, New
  York, 1969.
\newblock Republished in the ``books'' section of the S{\'e}minaire
  Lotharingien de Combinatoire.

\bibitem{Connes-Kreimer:9808042}
{\sc Alain Connes {\rm and }Dirk Kreimer}.
\newblock {\em Hopf algebras, renormalization and noncommutative geometry}.
\newblock Comm. Math. Phys. {\bf 199} (1998), 203--242.
\newblock \arxiv{hep-th/9808042}.

\bibitem{Dawson-Pare-Pronk:MR2116323}
{\sc Robert J.~MacG. Dawson, Robert Par{\'e}, {\rm and }Dorette~A. Pronk}.
\newblock {\em Universal properties of {S}pan}.
\newblock Theory Appl. Categ. {\bf 13} (2004), 61--85.

\bibitem{Dur:1986}
{\sc Arne D{\"u}r}.
\newblock {\em M\"obius functions, incidence algebras and power series
  representations}, vol. 1202 of Lecture Notes in Mathematics.
\newblock Springer-Verlag, Berlin, 1986.

\bibitem{Dyckerhoff-Kapranov:1212.3563}
{\sc Tobias Dyckerhoff {\rm and }Mikhail Kapranov}.
\newblock {\em Higher {S}egal spaces {I}}.
\newblock Preprint, \arxiv{1212.3563}, to appear in Lecture Notes in Mathematics.

\bibitem{Figueroa-GraciaBondia:0408145}
{\sc H{\'e}ctor Figueroa {\rm and }Jos{\'e}~M. Gracia-Bond{\'{\i}}a}.
\newblock {\em Combinatorial {H}opf algebras in quantum field theory. {I}}.
\newblock Rev. Math. Phys. {\bf 17} (2005), 881--976.
\newblock \arxiv{hep-th/0408145}.

\bibitem{Foissy:2002I}
{\sc Lo{\"\i}c Foissy}.
\newblock {\em Les alg\`ebres de {H}opf des arbres enracin\'es d\'ecor\'es.
  {I}}.
\newblock Bull. Sci. Math. {\bf 126} (2002), 193--239.

\bibitem{Foissy:MR3016301}
{\sc Lo{\"\i}c Foissy}.
\newblock {\em Algebraic structures on double and plane posets}.
\newblock J. Algebraic Combin. {\bf 37} (2013), 39--66.
\newblock \arxiv{1101.5231}.

\bibitem{Foissy:MR3046302}
{\sc Lo{\"\i}c Foissy}.
\newblock {\em Plane posets, special posets, and permutations}.
\newblock Adv. Math. {\bf 240} (2013), 24--60.
\newblock \arxiv{1109.1101}.


\bibitem{GalvezCarrillo-Kock-Tonks:1207.6404}
{\sc Imma G{\'a}lvez-Carrillo, Joachim Kock, {\rm and }Andrew Tonks}.
\newblock {\em Groupoids and {F}a{\`a} di {B}runo formulae for {G}reen
  functions in bialgebras of trees}.
\newblock Adv. Math. {\bf 254} (2014), 79--117.
\newblock \arxiv{1207.6404}.

\bibitem{GKT:1404.3202}
{\sc Imma G{\'a}lvez-Carrillo, Joachim Kock, {\rm and }Andrew Tonks}.
\newblock {\em Decomposition Spaces, Incidence Algebras and {M}{\"o}bius
  Inversion}.
\newblock (Old omnibus version, not intended for publication.)
Preprint, \arxiv{1404.3202}.

\bibitem{GKT:HLA}
{\sc Imma G{\'a}lvez-Carrillo, Joachim Kock, {\rm and }Andrew Tonks}.
\newblock {\em Homotopy linear algebra}.
\newblock To appear in Proc. Royal Soc. Edinburgh A.
\newblock \arxiv{1602.05082}.

\bibitem{GKT:DSIAMI-1}
{\sc Imma G{\'a}lvez-Carrillo, Joachim Kock, {\rm and }Andrew Tonks}.
\newblock {\em Decomposition spaces, incidence algebras and {M}{\"o}bius
  inversion I: basic theory}.
\newblock To appear in Adv. Math.
\newblock \arxiv{1512.07573}.

\bibitem{GKT:DSIAMI-2}
{\sc Imma G{\'a}lvez-Carrillo, Joachim Kock, {\rm and }Andrew Tonks}.
\newblock {\em Decomposition spaces, incidence algebras and {M}{\"o}bius
  inversion II: completeness, length filtration, and finiteness}.
\newblock Preprint, \arxiv{1512.07577}.

\bibitem{GKT:MI}
{\sc Imma G{\'a}lvez-Carrillo, Joachim Kock, {\rm and }Andrew Tonks}.
\newblock {\em Decomposition spaces, incidence algebras and M\"obius inversion
  III: the decomposition space of {M}{\"o}bius intervals}.
\newblock To appear in Adv. Math.
\newblock \arxiv{1512.07580}.

\bibitem{GKT:ex}
{\sc Imma G{\'a}lvez-Carrillo, Joachim Kock, {\rm and }Andrew Tonks}.
\newblock {\em Decomposition spaces in combinatorics}.
\newblock Preprint, \arxiv{1612.09225}.

\bibitem{Gessel:84}
{\sc Ira~M. Gessel}.
\newblock {\em Multipartite {$P$}-partitions and inner products of skew {S}chur
  functions}.
\newblock In {\em Combinatorics and algebra ({B}oulder, CO, 1983)}, vol.~34 of
  Contemp. Math., pp. 289--317. Amer. Math. Soc., Providence, RI, 1984.

\bibitem{Hermida:repr-mult}
{\sc Claudio Hermida}.
\newblock {\em Representable multicategories}.
\newblock Adv. Math. {\bf 151} (2000), 164--225.

\bibitem{Humpert-Martin}
{\sc Brandon Humpert {\rm and }Jeremy~L. Martin}.
\newblock {\em The incidence {H}opf algebra of graphs}.
\newblock SIAM J. Discrete Math. {\bf 26} (2012), 555--570.
\newblock \arxiv{1012.4786}.

% \bibitem{Illusie1}
% {\sc Luc Illusie}.
% \newblock {\em Complexe cotangent et d\'eformations. {I}}.
% \newblock No. 239 in Lecture Notes in Mathematics. Springer-Verlag, Berlin,
%   1971.

\bibitem{Illusie2}
{\sc Luc Illusie}.
\newblock {\em Complexe cotangent et d\'eformations. {II}}.
\newblock No. 283 in Lecture Notes in Mathematics. Springer-Verlag, Berlin,
  1972.

\bibitem{JoniRotaMR544721}
{\sc Saj-nicole~A. Joni {\rm and }Gian-Carlo Rota}.
\newblock {\em Coalgebras and bialgebras in combinatorics}.
\newblock Stud. Appl. Math. {\bf 61} (1979), 93--139.

\bibitem{JoyalMR633783}
{\sc Andr{\'e} Joyal}.
\newblock {\em Une th\'eorie combinatoire des s\'eries formelles}.
\newblock Adv. Math. {\bf 42} (1981), 1--82.

\bibitem{Joyal:disks}
{\sc Andr{\'e} Joyal}.
\newblock {\em Disks, duality and {$\Theta$}-categories}, September 1997.

\bibitem{Kock:0807}
{\sc Joachim Kock}.
\newblock {\em Polynomial functors and trees}.
\newblock Int. Math. Res. Notices {\bf 2011} (2011), 609--673.
\newblock \arxiv{0807.2874}.

\bibitem{Kock:MFPS28}
{\sc Joachim Kock}.
\newblock {\em Data types with symmetries and polynomial functors over
  groupoids}.
\newblock In {\em Proceedings of the 28th Conference on the Mathematical
  Foundations of Programming Semantics (Bath, 2012)}, vol. 286 of Electr. Notes
  Theoret. Comput. Sci., pp. 351--365, 2012.
\newblock \arxiv{1210.0828}.

\bibitem{Kock:1109.5785}
{\sc Joachim Kock}.
\newblock {\em Categorification of {H}opf algebras of rooted trees}.
\newblock Cent. Eur. J. Math. {\bf 11} (2013), 401--422.
\newblock \arxiv{1109.5785}.

\bibitem{Kock:1411.3098}
{\sc Joachim Kock}.
\newblock {\em Perturbative renormalisation for not-quite-connected
  bialgebras}.
\newblock Lett. Math. Phys. {\bf 105} (2015), 1413--1425.
\newblock \arxiv{1411.3098}.

\bibitem{Kock:1407.3744}
{\sc Joachim Kock}.
\newblock {\em Graphs, hypergraphs, and properads}.
\newblock Collect. Math. {\bf 67} (2016), 155--190.
\newblock \arxiv{1407.3744}.

\bibitem{Kock:1512.03027}
{\sc Joachim Kock}.
\newblock {\em Polynomial functors and combinatorial {D}yson--{S}chwinger
  equations}.
\newblock J. Math. Phys. {\bf 58} (2017), 041703, 36pp.
\newblock \arxiv{1512.03027}.

\bibitem{Kock-Weber:1609.03276}
{\sc Joachim Kock {\rm and }Mark Weber}.
\newblock {\em Fa\`a di {B}runo for operads and internal algebras}.
\newblock Preprint, \arxiv{1609.03276}.

\bibitem{LawvereMenniMR2720184}
{\sc F.~William Lawvere {\rm and }Mat{\'i}as Menni}.
\newblock {\em The {H}opf algebra of {M}\"obius intervals}.
\newblock Theory Appl. Categ. {\bf 24} (2010), 221--265.

\bibitem{Leroux:1975}
{\sc Pierre Leroux}.
\newblock {\em Les cat{\'e}gories de {M}{\"o}bius}.
\newblock Cahiers Topol. G{\'e}om. Diff. {\bf 16} (1976), 280--282.

\bibitem{Lurie:HA}
{\sc Jacob Lurie}.
\newblock {\em Higher Algebra}.
\newblock Available from \url{http://www.math.harvard.edu/~lurie/}, 2013.

\bibitem{Malvenuto-Reutenauer:0905.3508}
{\sc Claudia Malvenuto {\rm and }Christophe Reutenauer}.
\newblock {\em A self paired {H}opf algebra on double posets and a
  {L}ittlewood-{R}ichardson rule}.
\newblock J. Combin. Theory Ser. A {\bf 118} (2011), 1322--1333.
\newblock \arxiv{0905.3508}.

\bibitem{Manchon:MR2921530}
{\sc Dominique Manchon}.
\newblock {\em On bialgebras and {H}opf algebras of oriented graphs}.
\newblock Confluentes Math. {\bf 4} (2012), 1240003, 10pp.
\newblock \arxiv{1011.3032}.

\bibitem{Manin:MR2562767}
{\sc Yuri~I. Manin}.
\newblock {\em A course in mathematical logic for mathematicians}, vol.~53 of
  Graduate Texts in Mathematics.
\newblock Springer, New York, second edition, 2010.
\newblock Chapters I--VIII translated from the Russian by Neal Koblitz, With
  new chapters by Boris Zilber and the author.

\bibitem{Manin:0904.4921}
{\sc Yuri~I. Manin}.
\newblock {\em Renormalization and computation {I}: motivation and background}.
\newblock In {\em O{PERADS} 2009}, vol.~26 of S\'emin. Congr., pp. 181--222.
  Soc. Math. France, Paris, 2013.
\newblock \arxiv{0904.4921}.

\bibitem{Mitchell:RSO}
{\sc Barry Mitchell}.
\newblock {\em Rings with several objects}.
\newblock Adv. Math. {\bf 8} (1972), 1--161.

\bibitem{Oxley}
{\sc James~G. Oxley}.
\newblock {\em Matroid Theory}.
\newblock Oxford Graduate Texts in Mathematics. Oxford University Press, 1997.

\bibitem{Rota:Moebius}
{\sc Gian-Carlo Rota}.
\newblock {\em On the foundations of combinatorial theory. {I}. {T}heory of
  {M}\"obius functions}.
\newblock Z. Wahrscheinlichkeitstheorie und Verw. Gebiete {\bf 2} (1964),
  340--368.

\bibitem{Schmitt:hacs}
{\sc William~R. Schmitt}.
\newblock {\em Hopf algebras of combinatorial structures}.
\newblock Canad. J. Math. {\bf 45} (1993), 412--428.

\bibitem{Schmitt:1994}
{\sc William~R. Schmitt}.
\newblock {\em Incidence {H}opf algebras}.
\newblock J. Pure Appl. Algebra {\bf 96} (1994), 299--330.

% \bibitem{Segal:1973}
% {\sc Graeme Segal}.
% \newblock {\em Configuration-spaces and iterated loop-spaces}.
% \newblock Invent. Math. {\bf 21} (1973), 213--221.

\bibitem{Stanley:Mem1972}
{\sc Richard~P. Stanley}.
\newblock {\em Ordered structures and partitions}.
\newblock Memoirs of the American Mathematical Society, No. 119.
\newblock American Mathematical Society, Providence, R.I., 1972.

\bibitem{Street:LNM420}
{\sc Ross Street}.
\newblock {\em Fibrations and {Y}oneda's lemma in a {$2$}-category}.
\newblock In {\em Category {S}eminar ({P}roc. {S}em., {S}ydney, 1972/1973)},
  pp. 104--133. Lecture Notes in Math., Vol. 420. Springer, Berlin, 1974.

\bibitem{Street:fibsinbicats}
{\sc Ross Street}.
\newblock {\em Fibrations in bicategories}.
\newblock Cahiers Topol. G\'eom. Diff. {\bf 21} (1980), 111--160.

\bibitem{Walde:1611.08241}
{\sc Tashi Walde}.
\newblock {\em Hall monoidal categories and categorical modules}.
\newblock Preprint, \arxiv{1611.08241}.

\bibitem{Weber:TAC13}
{\sc Mark Weber}.
\newblock {\em Generic morphisms, parametric representations and weakly
  {C}artesian monads}.
\newblock Theory Appl. Categ. {\bf 13} (2004), 191--234.

\bibitem{Weber:TAC18}
{\sc Mark Weber}.
\newblock {\em Familial 2-functors and parametric right adjoints}.
\newblock Theory Appl. Categ. {\bf 18} (2007), 665--732.

\bibitem{Young:1611.09234}
{\sc Matthew~B. Young}.
\newblock {\em Relative 2-Segal spaces}.
\newblock Preprint, \arxiv{1611.09234}.

\end{thebibliography}

\end{document}